\documentclass[12pt,oneside,reqno]{amsart}
\usepackage{amssymb,amsfonts,amsmath}
\usepackage{epsfig}
\usepackage{graphicx}
\usepackage{color}
\usepackage{anysize}

\marginsize{2cm}{2cm}{2cm}{3cm}

\begin{document}
\baselineskip=14pt

\newtheorem{defin}{Definition}[section]
\newtheorem{Prop}{Proposition}
\newtheorem{teo}{Theorem}[section]
\newtheorem{ml}{Main Lemma}
\newtheorem{con}{Conjecture}
\newtheorem{cond}{Condition}
\newtheorem{conj}{Conjecture}
\newtheorem{prop}[teo]{Proposition}
\newtheorem{lem}{Lemma}[section]
\newtheorem{rmk}[teo]{Remark}
\newtheorem{cor}{Corollary}[section]
\renewcommand{\theequation}{\thesection .\arabic{equation}}

\newcommand{\beq}{\begin{equation}}
\newcommand{\eeq}{\end{equation}}
\newcommand{\beqn}{\begin{eqnarray}}
\newcommand{\beqnn}{\begin{eqnarray*}}
\newcommand{\eeqn}{\end{eqnarray}}
\newcommand{\eeqnn}{\end{eqnarray*}}
\newcommand{\bprop}{\begin{prop}}
\newcommand{\eprop}{\end{prop}}
\newcommand{\bteo}{\begin{teo}}
\newcommand{\bcor}{\begin{cor}}
\newcommand{\ecor}{\end{cor}}
\newcommand{\bcon}{\begin{con}}
\newcommand{\econ}{\end{con}}
\newcommand{\bcond}{\begin{cond}}
\newcommand{\econd}{\end{cond}}
\newcommand{\bconj}{\begin{conj}}
\newcommand{\econj}{\end{conj}}
\newcommand{\eteo}{\end{teo}}
\newcommand{\brm}{\begin{rmk}}
\newcommand{\erm}{\end{rmk}}
\newcommand{\blem}{\begin{lem}}
\newcommand{\elem}{\end{lem}}
\newcommand{\ben}{\begin{enumerate}}
\newcommand{\een}{\end{enumerate}}
\newcommand{\bei}{\begin{itemize}}
\newcommand{\eei}{\end{itemize}}
\newcommand{\bdf}{\begin{defin}}
\newcommand{\edf}{\end{defin}}
\newcommand{\bpr}{\begin{proof}}
\newcommand{\epr}{\end{proof}}

\newcommand{\halmos}{\rule{1ex}{1.4ex}}
\def \qed {{\hspace*{\fill}$\halmos$\medskip}}

\newcommand{\fr}{\frac}
\newcommand{\Z}{{\mathbb Z}}
\newcommand{\R}{{\mathbb R}}
\newcommand{\E}{{\mathbb E}}
\newcommand{\C}{{\mathbb C}}
\renewcommand{\P}{{\mathbb P}}
\newcommand{\N}{{\mathbb N}}
\newcommand{\var}{{\mathbb V}}
\renewcommand{\S}{{\mathcal S}}
\newcommand{\T}{{\mathcal T}}
\newcommand{\W}{{\mathcal W}}
\newcommand{\X}{{\mathcal X}}
\newcommand{\Y}{{\mathcal Y}}
\newcommand{\h}{{\mathcal H}}
\newcommand{\f}{{\mathcal F}}
\newcommand{\ex}{{\mathcal E}}

\renewcommand{\a}{\alpha}
\renewcommand{\b}{\beta}
\newcommand{\g}{\gamma}
\newcommand{\G}{\Gamma}
\renewcommand{\L}{\Lambda}
\renewcommand{\l}{\lambda}
\renewcommand{\d}{\delta}
\newcommand{\D}{\Delta}
\newcommand{\e}{\epsilon}
\newcommand{\s}{\sigma}
\newcommand{\B}{{\mathcal B}}
\renewcommand{\o}{\omega}

\newcommand{\nn}{\nonumber}
\renewcommand{\=}{&=&}
\renewcommand{\>}{&>&}
\newcommand{\<}{&<&}
\renewcommand{\le}{\leq}
\newcommand{\+}{&+&}

\newcommand{\pa}{\partial}
\newcommand{\ffrac}[2]{{\textstyle\frac{{#1}}{{#2}}}}
\newcommand{\dif}[1]{\ffrac{\partial}{\partial{#1}}}
\newcommand{\diff}[1]{\ffrac{\partial^2}{{\partial{#1}}^2}}
\newcommand{\difif}[2]{\ffrac{\partial^2}{\partial{#1}\partial{#2}}}

\parindent=0cm
\def\8{\infty}
\def\la{\lambda}
\newcommand{\eps}{\varepsilon}
\newcommand{\cvlaw}{\stackrel{\rm law}{\longrightarrow}}
\newcommand{\eqlaw}{\stackrel{\rm law}{=}}
\newcommand{\uu}{{\tt u}}

\title{
Last passage percolation  and traveling fronts 
}


\author{Francis Comets$^{1,4}$, Jeremy Quastel$^{2}$ and Alejandro F. Ram\'\i rez$^{3,4}$}

\thanks{ AMS 2000 {\it subject classifications}. Primary  
60K35, 82C22;
 secondary 
 60G70, 82B43.}


\thanks{$^1$Universit\'e Paris Diderot. Partially supported by CNRS, Laboratoire de Probabilit\'es et Mod\`eles
 Al\'eatoires,
UMR 7599.}

\thanks{$^2$University of Toronto. Partially supported by NSERC, Canada}

\thanks{$^3$Pontificia Universidad Cat\'olica de Chile. Partially supported by Fondo Nacional de Desarrollo Cient\'\i fico
y Tecnol\'ogico grant 1100298}

\thanks{$^4$Partially supported by ECOS-Conicyt grant CO9EO5}

\address[Francis Comets]{Universit\'e Paris Diderot - Paris 7\\
Math\'ematiques, case 7012\\
F-75 205 Paris Cedex 13, France}
\email{comets@math.univ-paris-diderot.fr}

\address[Jeremy Quastel]{Departments of Mathematics and Statistics\\
University of Toronto\\
40 St. George Street\\
Toronto, Ontario M5S 1L2, Canada}
\email{quastel@math.toronto.edu}
\address[Alejandro F. Ram\'\i rez]{Facultad de Matem\'aticas\\
Pontificia Universidad Cat\'olica de Chile\\
Vicu\~na Mackenna 4860, Macul\\
Santiago, Chile}

\email{aramirez@mat.puc.cl}
\bigskip

\date{January 28, 2013}
\maketitle


\begin{abstract} We consider a system of $N$ particles with a stochastic 
dynamics introduced by Brunet and Derrida \cite{BD04}. The particles can be interpreted as last passage times in directed percolation on $\{1,\dots,N\}$ of mean-field type.
The particles remain grouped and  
move like a traveling 
front, subject to discretization and driven by a random noise. As $N$ increases, we obtain estimates for the speed of the front and its profile, for different laws of the driving noise. As shown in \cite{BD04}, the model
with Gumbel distributed jumps has a simple structure. We establish  that the scaling limit is a 
L\'evy process in this case.
We study other jump distributions. We  prove a result showing that the limit for large $N$ is stable under small perturbations of the Gumbel.
In the opposite case of bounded jumps,  a completely different behavior is found, where  finite-size corrections are extremely small.
\end{abstract} 


\section{Definition of the model } \label{sec:def}
We consider the following stochastic process introduced by Brunet 
and Derrida \cite{BD04}. It consists in a fixed number $N \geq 1$
of particles on the real line, 
initially  at the positions
$X_1(0),\ldots, X_N(0)$. With $\{ \xi_{i,j}(s):1\le i,j\le N, s \geq 1\}$ an  i.i.d.~family of real random 
variables, the positions evolve as
\begin{equation}
  \label{eq:defXi}
  X_i(t+1)=\max_{1\le j\le N}\big\{X_j(t)+\xi_{i,j}(t+1)\big\}.
\end{equation}
The components of the $N$-vector $X(t)=(X_i(t), 1 \leq i \leq N)$
are not ordered.
The vector $X(t)$
 describes the location 
after the $t$-th step of a population under reproduction, mutation and selection keeping
the size constant. Given the current positions of the population, the next  positions
are a $N$-sample of the maximum of the full set of previous ones evolved by an independent step.
It can be also viewed as long-range directed polymer 
in random medium with $N$ sites in the transverse direction,
\begin{equation}  \label{eq:Xipolymer}
  X_i(t)=\max \big\{X_{j_0}(0)+ \sum_{s=1}^{t}\xi_{j_{s},j_{s-1}}(s);
1\le j_s \le N \;\forall s =0,\ldots t-1, j_{t}=i
\big\},
\end{equation}
as can be checked by induction ($1\leq i \leq N$). The model is long-range since the maximum in 
(\ref{eq:defXi}) ranges over all $j$'s. For comparison with a short-range model,
taking  the maximum over $j$ neighbor of $i$
in $\mathbb Z$ in (\ref{eq:defXi})  would define the standard oriented last passage percolation model with
passage time $\xi$ on edges in two dimensions.
\medskip

By the selection mechanism, the $N$ particles remain grouped even when $N \to \infty$, they are 
essentially pulled by the leading ones,
and the global motion is similar to a front propagation in reaction-diffusion equations
with traveling waves. Two ingredients of major interest are:
(i) the discretization effect of a finite $N$, (ii) the presence of
a random noise in the evolution. Such fronts are of interest, but poorly understood; see \cite{panja}
for a survey from a physics perspective.
\medskip

Traveling fronts appear in mean-field models for random growth. This was discovered 
by Derrida and Spohn \cite{DeSp88} for directed polymers in random medium on the tree,
and then extended to other problems \cite{MajKrap02, MonthusGarel}. 

\medskip

The present model was introduced by Brunet and Derrida in \cite{BD04} to compute the corrections
for large but finite system size to some continuous limit equations in front propagation.  
Corrections are due to finite size, quantization and stochastic effects.
They predicted, for a large class of such models  where the front is pulled by the farmost particles \cite{BD04, BDMM06}, that the motion and the particle structure have universal features, depending on just a few parameters related to the upper tails. 
Some of these predictions have been rigorously proved in specific contexts, such as the corrections to the speed of the Branching Random Walk 
(BRW) under the effect of a  selection \cite{BeGo10}, of the solution to  KPP equation with a small stochastic noise \cite{MuMyQu11}, or the genealogy of  branching Brownian motions with selection
\cite{BeBeSc11}.
For the so-called $N$-BBM (branching Brownian motion with killing of leftmost particles 
to keep the population size constant and equal to $N$)
the renormalized  fluctuations for the position of the killing barrier converge to a Levy process as $N$ diverges  \cite{Ma11b}.

\medskip

We mention other related references. For a continuous-time model with mutation and selection  conserving the total mass, the empirical measure converges to a free boundary problem
with a convolution kernel \cite{DR10}. Traveling waves are given by a Wiener-Hopf equation.
For a different model mimicking competition between infinitely many competitors,  called Indy-500,
quasi-stationary probability measures for competing particles seen from the leading edge  corresponds to a superposition of Poisson processes \cite{AizRuz04}. For diffusions interacting through their rank,
the spacings are tight \cite{PalPit08}, and the self-normalized exponential converge to a Poisson-Dirichlet  law \cite{CP10}.
In \cite{BRT}, particles  jump forward at a rate depending on their relative position with respect to the 
center of mass, with a higher rate for the particle behind: convergence to a traveling front is proved, 
which is given in some cases by the Gumbel distribution.

\medskip

We now give a flavor  of our results.
The Gumbel law ${\rm G}(0,1)$ has  distribution function 
$\P(\xi \leq x)= \exp\ -e^{-x}, x \in \R$. In \cite{BD04} it is shown that an appropriate measure
of the front location of a state $X \in \R^N$ in this case is
$$
\Phi(X)= \ln \sum_{1 \leq j \leq N} e^{X_j}\;,
$$
and that $\Phi(X(t))$ is a random walk, a feature which simplifies the analysis.
For an arbitrary distribution of $\xi$, the speed of the front with $N$ particles can be defined as the almost sure limit
$$
v_N = \lim_{t \to \8} t^{-1} \Phi(X(t))\;.
$$
We emphasize that $N$ is fixed in the previous formula, though it is sent to infinity in the next 
result.
Our first result is the scaling limit as the number $N$ of particles diverges.
\begin{teo} \label{cor:mettre en intro}     
Assume  $\xi_{i,j}(t)  \sim {\rm G}(0,1)$. Then, for all sequences $m_N \to \8$ as $N \to \8$,
$$
\frac{ \Phi(X([m_N\tau])) - \beta_N m_N \tau }{m_N/\ln N}  \cvlaw \S(\tau)
$$
in the Skorohod topology with $\S(\cdot)$ a totally asymmetric Cauchy process with L\'evy
exponent $\psi_C$ from (\ref{eq:exponent}),
where
$$
\beta_N = \ln b_N + Nb_N^{-1} \ln m_N ,
$$
with $\ln b_N=   \ln N + \ln \ln N - \frac{\gamma}{\ln N} + {\mathcal O}(\frac{1}{\ln^2 N} )$, see 
(\ref{eq:b_N}).
\end{teo}
Fluctuations of the front location are Cauchy distributed in the large $N$ limit.
Keeping $N$ fixed, the authors in \cite{BD04} find that
they are asymptotically Gaussian as $t \to \8$. We prove here that, as $N$ is sent to infinity, they are stable with index 1, a fact which has been overlooked in \cite{BD04}. When large populations are considered, this
is  the relevant point of view.
The Cauchy limit also holds true in the boundary case when time is not speeded-up
($m_N=1$) and $N \to \8$. 
For most growth models,  finding the scaling limit is  notoriously difficult.  In the present  model, 
 it is not difficult for the Gumbel distribution, but remains an open question for any other distribution.
\medskip

We next consider the case when $\xi$ is a perturbation of the Gumbel law.
Define $\eps(x) \in [-\8,1]$ by
\begin{equation} \label{eq:eps}
 \eps(x)=1+e^x \ln {\P(\xi \leq x)} .
\end{equation}
Note that $\eps \equiv 0$ 
is the case of  $\xi \sim {\rm G}(0,1)$.
The empirical 
distribution function (more precisely, its complement to 1)
of the $N$-particle system (\ref{eq:defXi})
is the random function
\begin{equation}
  \label{def:empiric}
U_N(t,x) = N^{-1} \sum_{i=1}^N {\bf 1}_{X_i(t) > x}
\end{equation}
This  is a non-increasing step function with jumps of size $1/N$ and limits $U_N(t,-\8)=1, $
$U_N(t,+\8)=0$. It has the  shape of a front wave, propagating at mean speed $v_N$, 
and it combines two interesting aspects: randomness and discrete values. We will call it the 
front profile, and we study in the next result  its relevant part, around the front location.
\begin{teo}
\label{th:cvTW} Assume  that
\begin{equation} \label{eq:eps2}
\lim_{x \to +\8} \eps(x)=0, \qquad {\rm and} \quad  \eps(x) \in [-\delta^{-1},1-\delta], 
\end{equation}
for all $x$ and some $\delta  >0$.
 Then, for all initial configurations $X(0) \in \R^N$, all   
$k \geq 1$, all $K_N \subset \{1,\dots,N\}$ with cardinality $k$,
and all $t \geq 2$
we have
\begin{equation} \label{eq:cvTW1}
\Big(X_j(t) - \Phi(X(t-1)); j \in K_N\Big) 
\cvlaw G(0,1)^{\otimes k}, \quad N \to \8,
\end{equation}
with $\Phi$ from (\ref{eq:Phigumbel}), and moreover,
\begin{equation} \label{eq:cvTW2}
U_N\big(t,  \Phi(X(t\!-\!1))+x \big) \longrightarrow u(x)=1-e^{-e^{-x}} 
\end{equation}
uniformly in  probability as  $N \to \8$.
\end{teo}

As is well known, it is rare to find rigorous  perturbation results from exact computations for
such models. For example, the above  mentioned,  last passage oriented percolation model 
on the planar lattice,  
 is exactly sovable for  exponential passage times \cite{Joh00} or geometric ones \cite{BDJ99} on sites, and the fluctuations asympotically have a Tracy-Widom distribution.
 However, no  perturbative result has been obtained after a decade.  
 Even though our assumptions seem to be strong, it is somewhat surprising 
 that we can prove this result.
The second condition is equivalent to the following stochastic domination: there exist finite constants  $c<d$ ($c=\ln \delta, d=\ln(1+\delta^{-1})$) such that
\begin{equation} \label{eq:HIC}
g + c \leq_{\rm sto} \xi  \leq_{\rm sto} g + d, \qquad g \sim G(0,1).
\end{equation}•
This condition is reminiscent of assumption (1.13) in \cite{No11} used to control the fluctuations of the front location for KPP equation in random medium.
By Theorem \ref{th:cvTW}, as $N \to \8$, the front remains sharp and its profile, 
which is defined microscopically as the empirical 
distribution function of
particles, converges to the Gumbel distribution as $N \to \8$. 
Hence  the Gumbel distribution is not only stable, but it is also an attractor. 
\medskip


Finally, we study the finite-size corrections to the front speed in
a case when the distribution of $\xi$ is quite different from the Gumbel law. 

\begin{teo} 
\label{th:gap}
Let $b<a$ and $p \in (0,1)$, and assume that  the $\xi_{i,j}(t)$'s are integrable and satisfy
\begin{equation}
  \label{eq:abc}
  \P(\xi >a)=\P(\xi \in (b, a))=0, \qquad \P(\xi=a)=p,  \qquad \P(\xi \in (b-\eps,b])>0 
\end{equation}
for all $\eps >0$. Then, as $N \to \8$,
$$
  v_N =
a - (a-b) (1-p)^{N^2} 2^N + o\big((1-p)^{N^2} 2^N\big).
$$
\end{teo}

We note that in such a case, in the leading order terms of the expansion as $N \to \8$, the value of the speed depends only on a few features of the distribution of $\xi$: the largest value $a$, its
probability mass $p$ and the gap $a-b$ with second largest one. All these involve the top of the support of the distribution, the other details being irrelevant. Such a behavior is expected for pulled fronts. 
\medskip

Though the mechanisms are  different, we make a parallel between the model considered here, and the BRW with selection, in order to discuss the Brunet-Derrida correction of the front speed
$v_N$
with respect to its asymptotic value. For definiteness, denote by $\eta$
the displacement variable, assume that $\eta$ is a.s. bounded from above by a constant $a$, and assume the branching 
is constant and equal to ${\beta}>1$.
The results of \cite{BeGo10} are obtained 
for 
$\beta \times \P(\eta =a)<1$
(Assumption A3 together with Lemma 5 (3) in \cite{BeGo11}), 
resulting in a logarithmic correction: This case corresponds to the Gumbel distribution for $\xi$ in our model, e.g., to Theorems \ref{cor:mettre en intro}     and \ref{th:cvTW} .
In contrast, the assumptions of Theorem \ref{th:gap} yield a
much smaller correction (of order exponential of negative $N^2$). This other case 
corresponds for large $N$ to the assumption $\beta \times \P(\eta =a)>1$ for the BRW with selection, where the corrections are exponentially small \cite{CG13}, precisely given by
$\rho^N$ with $\rho<1$ the extinction probability of the supercritical Galton-Watson 
process of particles located at site $ta$ at time $t$. In our model, the branching number is
$N$ and $\rho$ is itself exponentially small, yielding the correct exponent of negative $N^2$, but not the factor $2^N$.

\medskip

The paper is organised as follows. Section 2 contains some standard facts for the model. Section 3 deals with the front location in the case of the Gumbel law for $\xi$. In Section 4, we study the asymptotics as $N \to \8$ of the front profile (for Gumbel law and small perturbations), and their
relations to traveling waves and reaction-diffusion equation. In Sections 5 and 6, we 
expand the speed in the case of  integer valued, bounded from above, $\xi$'s, starting with the Bernoulli case. Theorems \ref{cor:mettre en intro}  and \ref{th:gap}
are proved in Sections \ref{sec:scalingGumbel} and \ref{proof {th:gap}}
respectively.

\section{Preliminaries for fixed $N$} \label{sec:fixedN}

For any fixed $N$, we show here the existence of large time
asymptotics for the $N$-particles system. 
It is convenient to shift the whole 
system by the position of the leading particle, because we show that
there exists an invariant measure for the shifted process.

\medskip

{\bf The ordered process:}
We now consider the process $\tilde X=(\tilde X(t), t \in \N)$ obtained 
by ordering the components of $X(t)$ at each time $t$, i.e., the set
$\{\tilde X_{1}(t),\tilde X_{2}(t),
\ldots,\tilde X_N(t)\}$ coincides with 
$\{X_{1}(t),X_{2}(t),
\ldots, X_N(t)
\}$ and
$
\tilde X_{1}(t)\ge \tilde X_{2}(t)\ge\cdots \geq  \tilde X_{N}(t).
$
Then, $\tilde X$ is a Markov chain with state space 
$$\Delta_N:=\{ y \in \R^N: y_1 \geq y_2 \geq \ldots \geq y_N\}.$$
Given $\tilde X(t)$, the vector $X(t)$ is uniformly distributed on the
$N!$ permutations of $\tilde X(t)$. Hence, it is sufficient to study 
$\tilde X$ instead of $X$.
It is easy to see that the sequence  $\tilde X$ has the same 
law as $Y=(Y(t), t \geq 0)$, given by as a recursive sequence
\begin{equation}
  \label{eq:defY}
  Y(t+1)= {\rm ordered\ vector}
\Big(\max_{1\le j\le N}\big\{Y_j(t)+\xi_{i,j}(t+1)\big\}, 1 \leq i \leq N\Big).
\end{equation}
Note that, when $X(0)$ is not ordered, 
 $\tilde X(1)$  is not a.s. equal to $Y(1)$ starting from $Y(0)=\tilde X(0)$.
In this section we study the sequence $Y$, which is nicer than $\tilde X$ because of the recursion
(\ref{eq:defY}):
Denote by $T_{\xi(t+1)}$ the above mapping $Y(t) \mapsto Y(t+1)$
on $\Delta_N$, and observe first that
\begin{equation}
  \label{eq:T}
 Y(t) = T_{\xi(t)}\ldots  T_{\xi(2)}T_{\xi(1)} Y(0).
\end{equation}
For $y, x  \in \Delta_N$, write $y\leq x$ if $y_i\leq x_i$ for all $i \leq N$.
The mapping  $T_{\xi(t)}$ is monotone for the partial order on 
 $\Delta_N$, 
  i.e., for the solutions $Y, Y'$
of (\ref{eq:defY}) starting from $Y(0), Y'(0)$ we have
$$
 Y(0) \leq  Y'(0)  \Longrightarrow Y(t) \leq  Y'(t) ,
$$
and moreover, with ${\bf 1}=(1,1,\ldots,1)$, $r \in \R$ and $y \in \R^ N$,
\begin{equation}\label{eq:sa1}
T_{\xi(t)} (y + r {\bf 1}) = r {\bf 1}+ T_{\xi(t)} (y) .
\end{equation}

\medskip

{\bf The process seen from the leading edge:}
For each $x \in \R^N$, we consider its shift $x^0$ by the maximum,
$$
x^0_i = x_i - \max_{1\le j\le N}x_j,
$$
and the corresponding processes $X^0, Y^0$. We call $X^0, Y^0$,  the unordered 
process, respectively, the ordered process, 
seen from the leading edge. Note that
$T_{\xi(t)}(y^0)= T_{\xi(t)}(y)-(\max_j y_j) {\bf 1}$  by (\ref{eq:sa1}), which yields
$$
\Big( T_{\xi(t)}(y^0)\Big)^0=\Big( T_{\xi(t)}(y)\Big)^0;
$$
a similar relation holds for $x$'s instead of $y$'s. 
Then  $X^0, Y^0$ are Markov chains, with $Y^0$ taking 
values in 
$\Delta_N^0:=\{y \in \Delta_N: y_1=0\}$, and we denote by $\nu_t$
the law of $Y^0(t)$. 
\medskip

\begin{prop} 
\label{one}
There exists an unique invariant measure $\nu$ for the process $Y^0$ seen 
from the leading edge,
and we have
\begin{equation}
\label{three}
\lim_{t\to\infty}\nu_t=\nu.
\end{equation}
Furthermore, there exists a $\delta_N>0$ such that
\begin{equation}
\label{four}
||\nu_t-\nu||_{TV}\le (1-\delta_N)^t.
\end{equation}
\end{prop}
Similar results hold for the unordered process $X^0$, by the remark 
preceeding (\ref{eq:defY}). Also, we mention that the value of $\delta_N$ is not sharp.
\medskip

\begin{proof}
Consider the  random variable
$$\tau=\inf\big\{ t \geq 1: \xi_{i,1}(t)=\max\{\xi_{i,j}(t); j \leq N\}
\; \forall i \leq N\big\}.
$$
Then, $\tau$ is a stopping time for the filtration $({\mathcal F}_t)_{t \geq 0}$, with 
${\mathcal F}_t=\sigma \{\xi_{i,j}(s); s \leq t, i, j \geq 1\}$.
It is geometrically distributed with parameter not smaller than 
\begin{equation} \label{deltaN}
\delta_N = (1/N)^N.
\end{equation}
Denote by $\oplus, \ominus$ the configuration vectors
$$
\oplus=(0,0,\ldots,0), \qquad \ominus =(0,-\8,\ldots,-\8).
$$
They are extremal configurations in  (the completion of) $\Delta_N^0$ since
$\ominus \leq y \leq \oplus$ for all $y \in \Delta_N^0$.
Now, by definition of $\tau$ and  (\ref{eq:defY}),
$$
 T_{\xi(\tau)} \oplus =  T_{\xi(\tau)} \ominus = T_{\xi(\tau)} y \qquad
\forall y \in \Delta_N^0.
$$
Hence, for all $t \geq \tau$ and all $y \in \Delta_N$ such that 
$\max_{1\le j\le N}y_j=\max_{1\le j\le N}Y_j(0)$,
$$
Y(t) = T_{\xi(t)}\ldots  T_{\xi(2)}T_{\xi(1)} y .
$$
We can construct a renewal structure. Define $\tau_1=\tau$, and recursively
for $k \geq 0$, 
$\tau_{k+1}=\tau_k + \tau \circ \theta_{\tau_k}$ with $\theta$ the 
time-shift. This sequence is the success time sequence in a Bernoulli process,
we have $1 \leq \tau_1 < \tau_2 <\ldots<
\tau_k < \ldots < \8$ a.s. The following observation is plain but
fundamental.

\begin{lem}[Renewal structure] \label{lem:renewal}
The sequence $$(Y^0(s); 0 \leq s < \tau_1), (Y^0(\tau_1 \! + \! s);  0 \leq s < \tau_2 \!- \!\tau_1 ), 
 (Y^0(\tau_2 \! + \! s);  0 \leq s < \tau_3 \!- \!\tau_2 ), \ldots $$ is independent. 
Moreover,  for all $k \geq 1$, $ (Y^0(\tau_k \! + \! s);  s \geq 0)$ has the same law as 
$ (Y^0(1 \! + \! s); s \geq 0)$ starting from $ Y^0(0)= \oplus$. 
\end{lem}
\begin{proof} of Lemma \ref{lem:renewal}.
By the strong Markov property, the Markov chain $Y^0$ starts afresh from the stopping times
$\tau_1 < \tau_2 < \ldots$. This proves the first statement, and we now turn to the second one.
Note that $T_\xi \oplus = T_\eta \oplus $ if, for all $i$, $(\xi_{i,j}; j \leq N)$ is a permutation of
$(\eta_{i,j}; j \leq N)$. Hence,
$$
\P \big(  Y^0(1)   \in \cdot \; , \tau_1=1\;\vert Y^0(0)= \oplus \big)=
\P \big(  Y^0(1)  \in \cdot \;\vert Y^0(0)= \oplus \big) \times \P(\tau_1=1),
$$
and so
$$
\P \big(  Y^0(1)   \in \cdot \;\vert Y^0(0)= \oplus, \tau_1=1 \big)=
\P \big(  Y^0(1)  \in \cdot \;\vert Y^0(0)= \oplus \big).
$$
From the markovian structure and by induction it follows that
\begin{eqnarray*}
\P \big(  (Y^0(1 \! + \! s); s \geq 0 )  \in \cdot \;\vert Y^0(0)= \oplus \big)
&=&
\P \big(  (Y^0(1 \! + \! s); s \geq 0 )  \in \cdot \;\vert Y^0(0)= \oplus, \tau_1=1 \big)\\
&=&
\P \big(  (Y^0(1 \! + \! s); s \geq 0 )  \in \cdot \;\vert Y^0(0)= z, \tau_1=1 \big)\\
&=&
\P \big(  (Y^0(1 \! + \! s); s \geq 0 )  \in \cdot \;\vert  \tau_1=1 \big)\\
&=&
\P \big(  (Y^0(\tau_1 \! + \! s); s \geq 0 )  \in \cdot  \big),
\end{eqnarray*}
for all $z \in \Delta_N^0$. 
\end{proof}

The lemma implies the proposition, with the law $\nu$ given for a measurable $F: \Delta_N^0
\to \R_+$ by
\begin{eqnarray} \nonumber
\int F d\nu &=& \frac{1}{\E( \tau_2-\tau_1)}\; \E \sum_{\tau_1 \leq t < \tau_2}
F(Y^0(t)) \\ &=& \frac{1}{\E( \tau_1)}\; 
\sum_{t \geq 1} \E \big( F(Y^0(t)) {\mathbf 1}_{t<\tau_2} \vert 
\tau_1=1\big). \label{eq:mesinv}
\end{eqnarray}
\end{proof}

\begin{rmk} \label{rk:1}
(i) The proposition shows that the particles remain grouped as $t$ increases, i.e., the law of the distance 
between extreme  particles is a tight sequence under the time
evolution. In Theorem \ref{th:cvTW} we will see that when the law of $\xi$ is close to Gumbel, they remain
grouped too as $N$ increases.

\smallskip

(ii) The location of front 
at time $t$
can be described by any numerical function $\Phi(Y(t))$ or $\Phi(Y(t-1))$
(or equivalently, any  symmetric function of $X(t)$ or $X(t-1)$)
which commutes to space translations by constant vectors, 
\begin{equation}
\label{comm}
\Phi(y+ r {\bf 1})=r + \Phi(y)\;,
\end{equation}
and which is increasing for the partial order on $\R^N$.
Among such, we mention also the maximum or the minimum value, the arithmetic mean, the median or any other order
statistics, and the choice in (\ref{eq:Phigumbel}) below.
For Proposition \ref{one}, we have taken the first choice -- the 
location of the rightmost particle --
for simplicity. 
Some other choices may be more appropriate to describe the front, by
looking in the bulk of the system rather than at the leading edge. 
For fixed $N$ all such choices will however 
lead to the same value for the speed $v_N$ of the front, that we define below.
\end{rmk}
Note that for a function $\Phi$ which satisfies
the commutation relation (\ref{comm}) we have  the inequalities
$$
 \Phi( \oplus) + \min_{i \leq N} \max_{j \leq N} \left\{ Y_j(0)+ \xi_{i,j}(1)\right\} \leq
\Phi(Y(1)) \leq \Phi( \oplus) + \max_{i,j \leq N} \left\{ Y_j(0)+ \xi_{i,j}(1)\right\} .
$$
Now, by equation (\ref{eq:mesinv}) and by the fact that $\tau_1$ is
stochastically smaller than a geometric random variable with parameter
$(1/N)^N$ we  conclude that if $\xi\in L^p, Y(0) \in L^p$ then
 $\Phi(Y(t)) \in L^p$, and also
$\int |y_N|^p d\nu(y) < \8$. The following corollary is a straightforward
consequence of the above.

\begin{cor}[Speed of the front] \label{cor:speed}
If $\xi \in L^1$, the following limits
$$
v_N=\lim_{t \to \8} t^{-1} \max\{ X_i(t); 1 \leq i \leq N\}=
\lim_{t \to \8} t^{-1} \min\{ X_i(t); 1 \leq i \leq N\}
$$
exist a.s.,  and $v_N$   is given by
$$
v_N = \int_{\Delta_N^0} d\nu(y)\;
\E \max_{1\le i, j\le N}\big\{y_j+\xi_{i,j}(1)\big\}.
$$
Moreover, if $\xi \in L^2$, 
$$ 
t^{-1/2} \big( \max\{ X_i(t); 1 \leq i \leq N\}-v_N t \big)
$$ 
converges in law  as $t \to \8$
to a Gaussian r.v. with variance $\sigma_N^2 \in (0,\8)$.
\end{cor}
We call $v_N$ the speed of the front of the $N$-particle system. 

\begin{proof} 
The equality of the two limits in the definition of $v_N$ follows from 
tightness in Remark \ref{rk:1}, (i), and the existence is from the renewal structure.   
Similarly, we have
$$
v_N=\int_{\Delta_N^0}  \big(
\E \Phi(T_\xi y) - \Phi(y) \big)\;d\nu(y)
$$
for all $\Phi$ as in Remark \ref{rk:1}, (ii), where  $\Delta_N^0$ is defined just before Proposition \ref{one}.  The second formula is obtained by taking $\Phi(y)=\max_{i \leq N} y_i$.
The Gaussian limit is the Central Limit Theorem for renewal processes.
\end{proof}

%

\section{The Gumbel distribution}   \label{sec:gumbel}

 The Gumbel law G($a,\la$)
with scaling parameter $\lambda >0$ and location parameter $a \in \R$
is defined by its distribution function
\begin{equation} \label{eq:gumbel}
\P( \xi \leq x) = \exp \big( - e^{-\la (x-a)}\big), \qquad x \in \R.
\end{equation}
This law is known to be a limit law in
extreme value theory \cite{LeadbetterLR83}.
In \cite{BD04}, Brunet and Derrida considered
the standard case $a=0, \la=1$, to find a complete explicit solution
to the model. 
In this section, we 
assume that the sequence $\xi_{i,j}$ is G($a,\la$)-distributed, for some 
$a, \la$.
Then, $\zeta = \la (\xi -a) \sim$  G($0,1$), while
\begin{equation} \label{eq:propgumbel}
\exp(- \zeta) \quad {\rm is\ exponentially\  distributed\ with\  parameter\ }1,
\end{equation}
and $\exp ( - e^{-\zeta})$ is uniform on $(0,1)$. Conversely, if $U$ is uniform on $(0,1)$ and
$\ex$ exponential of parameter 1, then $-\la \ln \ln (1/U)$ and $\ln \ex^{-\la}$ are $G(0,\la)$. 
\medskip

Here, the Gumbel distribution makes the model stationary for fixed $N$ and allows exact computations.

\subsection{The Front as a random walk} \label{sec:rwgumbel}
In this section, we fix $N \geq 1, a \in \R, \lambda>0$.
We will choose the function $\Phi: \R^N \to \R$, 
\begin{equation} \label{eq:Phigumbel}
\Phi(x)= \la^{-1} \ln \sum_{i=1}^N \exp \la x_i 
\end{equation}
to describe the front location $\Phi(X(t))$ at time $t$.
\begin{teo}[\cite{BD04}]
\label{th:gumbel} Assume the $\xi_{i,j}$'s are Gumbel $G(a, \lambda)$-distributed.

(i) Then, the sequence $(\Phi(X(t)); t \geq 0)$ is a random walk, with increments 
\begin{equation} \label{eq:Ups}
\Upsilon=a + \la^{-1} \ln \left( \sum_{i=1}^N \ex_{i}^{-1} \right)
\end{equation}
where the $\ex_i$ are i.i.d. exponential of parameter 1.

(ii) Then,
\begin{equation} \label{eq:vgumbelala1}
v_N = a + \la^{-1} \E \ln \left( \sum_{i=1}^N \ex_{i}^{-1} \right),\qquad
\sigma^2_N = \la^{-2} {\rm Var}\left(\ln  \sum_{i=1}^N \ex_{i}^{-1} \right).
\end{equation}

(iii) The law $\nu$ from proposition \ref{one} is the law of the shift
$V^0 \in \Delta_N^0$ of the ordered vector $V$
obtained from  a $N$-sample from a Gumbel G($0, \la$).
\end{teo}

\begin{proof}:
Define $\f_t=\sigma( \xi_{i,j}(s), s \leq t, i, j \leq N)$, and 
$\ex_{i,j}(t)=\exp\{- \la (\xi_{i,j}(t) -a) \}$.   
By (\ref{eq:defXi}),
\begin{eqnarray} \nonumber
X_i(t+1)&=&\max_{1\le j\le N}\big\{X_j(t)+a- \la^{-1} \ln \ex_{i,j}(t+1)
\big\}\\ \label{eq:stst}
&=& a + \Phi(X(t)) - \la^{-1} \ln \ex_{i}(t+1) ,
\end{eqnarray}
where
$$
\ex_{i}(t+1) = \min_{1\le j\le N}\big\{\ex_{i,j}(t+1) 
e^{-\la X_j(t)}
\big\}e^{\la \Phi(X(t))}
, \qquad t \geq 0.
$$
Given $\f_t$, each variable $\ex_{i}(t+1)$ is 
exponentially  distributed with  parameter 1 by the standard stability 
property of the exponential law under independent minimum, and moreover, the whole vector 
$(\ex_{i}(t+1), i \leq N)$
is conditionnally independent. Therefore, this vector is  independent
of $\f_t$, and finally,
$$(\ex_{i}(t), 1 \leq i \leq N, t \geq 1) \quad {\rm  is \ independent 
\ and \ identically\  distributed}
$$
with   parameter 1,  exponential law. Hence,  the sequence
$$
\Upsilon(t)=a + \la^{-1} \ln \left( \sum_{i=1}^N \ex_{i}(t)^{-1} \right), \qquad t \geq 1,
$$
is i.i.d. with the same law as $\Upsilon$. 
Now, by (\ref{eq:stst}),
\begin{eqnarray*}
\Phi(X(t))&=&
\Phi(X(t-1)) + \Upsilon(t)\\
&=&
\Phi(X(0))+ \sum_{s=1}^t  \Upsilon(s)
\end{eqnarray*}
which shows that  $(\Phi(X(t)); t \geq 0)$ is a random walk. 
Thus, we obtain both (i) and (ii).

\medskip
From (\ref{eq:stst}), we see that the conditional law of $X(t+1)$ given $\f_t$
is the law of a $N$-sample from a Gumbel G($a+\Phi(X(t)), \la$). Hence,
$\nu$ is the law of the order statistics of 
a $N$-sample from a Gumbel G($0, \la$), shifted by the leading edge.
\end{proof}

We end this section with a remark. Observe that the other max-stable laws (Weibull and Frechet) do not yield exact computations for our model. Hence, the special role of 
 the Gumbel is not  due to the stability of that law under taking the maximum of i.i.d. sample,
but also to its behavior under shifts.
\subsection{Asymptotics for large $N$ } \label{sec:asGumbel}  
In this section we study the asymptotics as $N \to \8$ with a stable limit law. 
When $a=0$ and $\la=1$, Brunet and Derrida  \cite{BD04}  obtain the expansions
\begin{equation} \label{eq:vgumbel}
v_N=  \ln N + \ln \ln N + \frac{\ln \ln N}{ \ln N} +  \frac{1-\gamma}{ \ln N} + o( \frac{1}{ \ln N})
,
\end{equation}
\begin{equation} \label{eq:sigumbel}
\sigma_N^2= \frac{\pi^2}{3\ln N}+\ldots,
\end{equation}
by Laplace method for an integral representation of the Laplace transform of $\Upsilon$.
 We recover here the first terms of the expansions from the stable limit law, 
 in the streamline of 
our approach.

\medskip

We start to determine the correct scaling for the jumps of the random walk.
First, observe that $\ex^{-1}$ belongs to the  domain of normal attraction of a stable law of index 1.
Indeed,  the tails distribution is
$$\P(\ex^{-1} > x)=1-e^{-1/x}\sim  x^{-1}, \qquad x \to +\8.$$
Then, from e.g. Theorem 3.7.2 in \cite{Durrett},
\begin{equation} \label{eq:cvstable}
\S^{(N)}:=\frac{ \sum_{i=1}^N \ex_{i}^{-1} - b_N}{N}  \cvlaw \S_{},
\end{equation}
where  $b_N=N \E(\ex^{-1};\ex^{-1}<N)$, and $\S_{}$ is the totally  asymmetric
stable law of index $\alpha=1$, with characteristic function given for $u \in \R$ by
\begin{eqnarray}  \nonumber
\E e^{iu \S_{}}&=& 
\exp \left\{ \int_1^\8 (e^{iux}-1) \frac{dx}{x^2} + \int_0^1 (e^{iux}-1-iux) \frac{dx}{x^2}  
\right\} \\  \nonumber
&=& 
\exp \left\{ iCu - \frac{\pi}{2} |u| \big\{ 1+ i \frac{2}{\pi} {\rm sign}(u) \ln |u| \big\} \right\}\\
&=:& \exp \Psi_C(u), \label{eq:exponent}
\end{eqnarray}
for some real constant $C$ defined by the above equality.
By integration by parts, one can check that, as $N \to \8$,
\begin{equation} \label{eq:b_N}
b_N=N \int_{1/N}^\8 \frac{e^{-y}}{y}dy = N \big(\ln N - \gamma
+ \frac{1}{N} +{\mathcal O}(\frac{1}{N^2} ) \big),
\end{equation}
with $\gamma= - \int_0^\8 e^{-x} \ln x dx$ the Euler constant. Then,
$$
\ln b_N = \ln N + \ln \ln N - \frac{\gamma}{\ln N} + {\mathcal O}(\frac{1}{\ln^2 N} )
$$
We need to estimate 
\begin{eqnarray} \nonumber
\E  \ln \sum_{i=1}^N \ex_{i}^{-1}  - \ln b_N&=& \E \ln \big(1+ \frac{N}{b_N} \S_{}^{(N)} \big)\\
&=& \E \ln \big(1+ \frac{N}{b_N} \S_{} \big) + 
	{\mathcal O}((\frac{1}{\ln N})^{1-\delta}),
\label{eq:troppeu}
\end{eqnarray}
for all $\delta \in (0,1]$: indeed, since the moments of $\S^{(N)}$ of order $1-\delta/2$ are bounded
(Lemma 5.2.2 in \cite{IbLi}),
the sequence $  (\frac{b_N}{N})^{1-\delta} \big[\ln \big(1+ \frac{N}{b_N} \S_{}^{(N)} \big)-\ln \big(1+ \frac{N}{b_N} \S_{} \big)\big]$ is uniformly integrable, and it converges to 0.
A simple computation shows that
$$
\E \ln (1+ \eps \S_{}) = \int_1^\8 \ln(1+\eps y) \frac{dy}{y^2} (1+o(1)) + {\mathcal O}(\eps)
\sim \eps \ln (\eps^{-1}) 
$$
as $\eps \searrow 0$.  With $\eps=N/b_N$, we recover  the first 2 terms in the formula (\ref{eq:vgumbel}) for $v_N$. (If we could improve the error term in (\ref{eq:troppeu})
to $o( \ln \ln N / \ln N)$, we would get also the third term.)

With a similar computation,  we estimate as $N \to \8$
\begin{eqnarray*}
\sigma_N^2 &=& 
{\rm Var} \Big(
\ln \big(1+ \frac{N}{b_N} \S_{}^{(N)} \big) \Big) \\
&\sim& 
{\rm Var} \Big(
\ln \big(1+ \frac{N}{b_N} \S_{}\big) \Big) \\
&\sim& 
\E
\ln^2 \big(1+ \frac{N}{b_N} \S_{}\big)  \\
&\sim& 
\int_1^\8
\ln^2( 1+\frac{y}{\ln N})\frac{dy}{y^2} \\
&\sim& 
\int_0^\8
\ln^2( 1+\frac{y}{\ln N})\frac{dy}{y^2} \\
&=& \frac{C_0}{\ln N}\;, 
\end{eqnarray*}
with $C_0= \int_0^\8 \ln^2( 1+y) \frac{dy}{y^2} = \pi^2/3$. 

\subsection{Scaling limit for large $N$} \label{sec:scalingGumbel}

In this section, we let the  parameters $a, \la$ of the Gumbel depend on $N$,  and get stable law and process  as scaling limits for the walk:
In view of the above, we assume in this subsection  that $\xi_{i,j} \sim {\rm G}(a,\lambda)$ where $a=a_N$ and $\la=\la_N$ depend on $N$,
\begin{eqnarray} \label{eq:a_N}
\left\{
\begin{array}{cclcl}
\la_N&=&  \frac{\displaystyle N}{\displaystyle b_N} &\sim& \frac{\displaystyle 1}{\displaystyle \ln N}, \\
a_N&=& -C- \la_N^{-1} \ln( b_N)&=& -C- \ln^2 N - (\ln N)( \ln \ln N) +o(1), 
\end{array}•
\right.
\end{eqnarray} 
with the constant $C$ from (\ref{eq:exponent}).
Correspondingly, we write $$X=X^{(N)},
\Upsilon_N(t)=a_N + \la_N^{-1} \ln \left( \sum_{i=1}^N \ex_{i}(t)^{-1} \right).$$ Note that,
with $\S^{(N)}$ defined by the left-hand side of (\ref{eq:cvstable}),
we have by (\ref{eq:Ups}),
\begin{eqnarray} \nonumber
\Upsilon_N &=& 
\frac{1}{\la_N} \ln \Big( \sum_{i=1}^N \ex_{i}^{-1} \Big) -
C-\frac{1}{\lambda_N}\ln b_N\\  \nonumber
&=&  \frac{1}{\la_N} \ln \Big( 1 + \frac{N}{b_N}\S^{(N)} \Big) -C \\
& \cvlaw & \S_0, \label{eq:rwcv}
\end{eqnarray}
as $N \to \8$, where the stable variable $\S_0=\S-C$ has characteristic function  
$$\E \exp iu \S_0= \exp \Psi_0(u), \qquad 
\Psi_0(u)= - \frac{\pi}{2} |u| - i u \ln |u| 
$$ 
from the particular choice of $C$.
In words, with an appropriate renormalization as the system size increases, the instantaneous jump  
of the front converges to a stable law.
For all integer $n$ and independent 
copies $\S_{0,1}, \ldots \S_{0,n}$ of $\S_0$, we see that
$$
\frac{\S_{0,1}+ \ldots +\S_{0,n}}{n} -  \ln n \eqlaw \S_0
$$
from the characteristic function.
Consider the totally asymmetric Cauchy process $(\S_0(\tau); \tau \geq 0)$,
i.e.  the independent increment process with characteristic function
$$
\E \exp\{iu (\S_0(\tau)-\S_0(\tau'))\}= \exp \{(\tau-\tau') \Psi_0(u)\}, \qquad u \in \R, 0<\tau'<\tau.
$$
It is a  L\'evy process with  L\'evy measure $x^{-2}$ on $\R_+$,
it is not self-similar but it is stable in a wide sense: for all $\tau>0$,
$$
\frac{\S_0(\tau)}{\tau} -  \ln \tau \eqlaw \S_0(1)
$$
with $\S_0(1) \eqlaw \S_0$. We refer to \cite{Ber96} for a nice account on 
 L\'evy processes.
\medskip

We may speed up the time of the front propagation as well, say by a factor $m_N \to \8$ when $N \to \8$,
to get a continuous time description.
Then, we consider another scaling, and define for $\tau >0$,
\begin{eqnarray} \nonumber
\varphi_N(\tau) &=& \frac{ \Phi(X^{(N)}([m_N\tau]))- \Phi(X^{(N)}(0))}{m_N} - \tau \ln m_N\\
 \label{eq:rwsc} &=&  \frac{ \sum_{t=1}^{[m_N\tau]} \Upsilon_N(t)}{m_N} - \tau \ln m_N 
\end{eqnarray}
by theorem \ref{th:gumbel}.
Of course, this new centering can be viewed as an additional shift in the formula (\ref{eq:a_N}) for $a_N$.
 By (\ref{eq:rwcv}), the  characteristic function $\chi_N(u):=\E e^{iu \Upsilon_N}=  \exp \{ \Psi_0(u)(1+o(1))\}$,
where $o(1)$ depends on $u$ and tends to 0 as $N \to \8$. Then,
\begin{eqnarray*}
\E \exp\left\{ iu
 (\frac{ \sum_{t=1}^{[m_N\tau]} \Upsilon_N(t)}{m_N} - \tau \ln m_N )
\right\} 
&=&
\left( \chi_N(u/m_N)  \right) ^{[m_N\tau]}\exp\{ -i\frac{u[m_N\tau]}{m_N} \ln m_N\}\\
&\to & \exp \tau \Psi_0(u),
\end{eqnarray*}
as $N \to \8$, showing convergence at a fixed time $\tau$. In fact, convergence holds at the process level.
\begin{teo} 
\label{th:levy}
As $N \to \8$, the process $\varphi_N(\cdot)$ converges in law 
 in the Skorohod topology to the totally asymmetric Cauchy process $\S_0(\cdot)$.
\end{teo}
\begin{proof}
The process $ \varphi_N(\cdot)$  itself has independent increments. The result follows from general results
on triangular arrays of independent variables in the domain of attraction of a stable law, e.g. Theorems 2.1 and 3.2 in \cite{Jacod}.
\end{proof}

{\it Proof of Theorem \ref{cor:mettre en intro}}: 
Apply the previous  Theorem \ref{th:levy} after making the substitution $\zeta=\lambda_N(\xi-a_N)$. 
\qed

\section{The front profile as a  traveling wave} \label{sec:travelingwave}

Recall the front profile 
\begin{equation}
  \label{def:empiric}
U_N(t,x) = N^{-1} \sum_{i=1}^N {\bf 1}_{X_i(t) > x}
\end{equation}
which is a wave-like, random step  function, traveling at speed $v_N$. One can write some kind of 
Kolmogorov-Petrovsky-Piscunov
equation with noise (and discrete time) governing its evolution, see (7--10)   in \cite{BD04} and 
Proposition \ref{g-case}.
Let $F$ denote the distribution function of the $\xi$'s, 
$F(x)=\P(\xi_{i,j}(t) \leq x)$. 
Given $\f_{t-1}$, the right-hand side is, up to the factor $N^{-1}$,
a binomial variable with parameters $N$ and
\begin{eqnarray} \nonumber
\P( X_i(t) >x \vert \f_{t-1}) 
&=&
1- \prod_{j=1}^N \P( X_j(t-1)+\xi_{i,j}(t) \leq x \vert \f_{t-1})
\qquad ({\rm by\ } (\ref{eq:defXi}))\\
&=&
1- \exp - N \int_\R 
 \ln F(x-y) U_N(t-1, dy). \label{eq:antoine}
\end{eqnarray}

\subsection{Gumbel case  } \label{sec:travelingwaveGumbel}

Starting with the  case of the Gumbel law $F(x)=\exp-e^{-\la(x-a)}$, 
we observe that (\ref{eq:antoine}) and (\ref{eq:Phigumbel}) imply 
$$
\P( X_i(t) \leq x \vert \f_{t-1}) 
=\exp  - e^{\la (x-a-\Phi(X(t-1)))},
$$
that is  (\ref{eq:stst}). 
It means that $X(t)-\Phi(X(t-1))$ is independent of ${\mathcal F}_{t-1}$, and that it is 
 a $N$-sample of the law $G(a,\la)$. 
For the  process at time $t$ centered by the front location $\Phi(X(t-1))$, the product measure $G(a,\la)^{\otimes N}$ is invariant. We summarize these observations:
\begin{prop}[\cite{BD04}]
\label{g-case}
Let  $\xi_{i,j}(t) \sim {\rm G}(a,\la)$ be given, and  $X$  defined by (\ref{eq:defXi}). 
Then, the random variables $G_i(t)$ defined by  $G(t)=(G_i(t); i \leq N)$ and
\begin{equation} \label{eq:rrs}
X(t)=G(t)+ \Phi(X(t-1)) {\bf 1}, \quad t \geq 1,
\end{equation}
are i.i.d.~with common law  $ {\rm G}(a,\la)$, and $G(t)$ is independent of $X(t-1),X(t-2),\ldots$. In particular,
$(G_i(t); i \leq N, t\geq 1)$ is an i.i.d. sequence with law G$(a,\la)$, independent of $X(0) \in \R^N$. Moreover,
\begin{equation} \label{eq:rrs2}
U_N(t,x)= \frac{1}{N} \sum_{i=1}^N {\bf 1}\{G_i(t)\geq x-\Phi(X(t\!-1\!))\}
, \quad t \geq 1, x \in \R.
\end{equation}
\end{prop}
\begin{rmk} (i) The recursion (\ref{eq:rrs2})  is the {\rm reaction-diffusion equation} satisfied by $U_N$. This equation is discrete and driven by a random noise $(G(t); t \geq 0)$.\\

(ii)
Note that the centering is given by a function of the configuration at the previous time $t-1$. One could easily get an 
invariant measure with a centering depending on the current configuration. For instance, consider
$$
X(t)- \max_j X_j(t)  \eqlaw g - \max_j g_j,
$$
with $g_j$ i.i.d. $G(a,\la)$-distributed, or replace the maximum value by another order statistics. However our centering, allowing interesting properties like the representation (\ref{eq:rrs}), is 
the most natural.
\end{rmk}

 By the law of large 
numbers,  as $N \to \8$,
the centered front converges almost surely to a limit front, given by  the (complement of)
the distribution function of G$(a, \la)$, as we state now.
\begin{prop} \label{prop:cvRD}
For all $t \geq 1$,   the following   holds:\\

(i) Convergence of the front profile: as $N \to \8$, conditionally on ${\mathcal F}_{t-1}$, we have a.s.
$$U_N\big(t,x+\Phi(X(t\!-\!1))\big) \longrightarrow u(x)= 1-\exp(-e^{-\la (x-a)}), \qquad {\rm uniformly\ in\ }  x \in \R.$$

(ii) Fluctuations: as $N \to \8$,
\begin{equation*}
\ln N \times \left\{
U_N\big(t,x+(t\!-\!1)(\ln b_N+a) +\Phi(X(0))\big)-u(x)\right\}
\cvlaw \frac{u'(x)}{\la} \big( t {\mathcal S} +t \ln t + tC\big)
\end{equation*}
as $N \to \8$, with $\mathcal S$ from (\ref{eq:cvstable}) and $C$ from (\ref{eq:exponent}).
\end{prop}
We willl see in the proof that the front location alone is responsible for the fluctuations of the profile. It dominates a 
smaller Gaussian fluctuation due to the sampling.\medskip


\begin{proof}of Proposition \ref{prop:cvRD}. As mentioned above, the law of large numbers yields  pointwise convergence in the first claim. 
Since $U_N(t,\cdot)$ is non-inceasing, uniformity follows from Dini's theorem. We now prove the fluctuation result.
 By (\ref{eq:Ups}) and (\ref{eq:cvstable}), 
$$
Z_N:=\ln N \times \big\{\Phi(X(t))-\Phi(X(0)) -t \ln b_N\big\} = \frac{\ln N}{\la}\sum_{s=1}^t  
\ln \big(1+ \frac{N}{b_N} \S_{}^{(N)}(s) \big)
$$
converges in law to the sum of $t$ independent copies of $\mathcal S$, which has itself the law of  
$ t {\mathcal S} +t \ln t + tC$. On the other hand, we have by (\ref{eq:rrs2}),
$$
U_N\big(t+1,x+t \ln b_N + \Phi(X(0))\big)= \frac{1}{N} \sum_{i=1}^N {\bf 1}\{G_i(t+1)\geq x+ \fr{1}{\ln N} Z_N\}.
$$
 By  the central limit theorem for triangular arrays, for all sequences $z_N \to 0$, we see that,
$$
N^{1/2}\big(  \frac{1}{N} \sum_{i=1}^N {\bf 1}\{G_i(t+1)\geq x+ z_N\} -u(x+z_N)\big) \cvlaw
Z \sim {\mathcal N}(0, u(x)(1\!-\!u(x)))
$$
as $N \to \8$. Being of order $N^{-1/2}$, these fluctuations will vanish in front of the Cauchy ones, which are of order $(\ln N)^{-1}$.
In the left hand side, we Taylor expand $u(x+z_N)$. Since $G(t+1)$ and $Z_N$ are independent, 
we obtain
$$
\ln N \times \left\{U_N\big(t+1,x+t\ln b_N +\Phi(X(0))\big)-u(x)\right\}-u'(x) Z_N \to 0 
$$
in probability, which proves the result. 
\end{proof}

\medskip

{\bf Remark}:  A limiting reaction-diffusion equation.
It is natural to look for a reaction-diffusion equation which has $u$ as traveling wave (soliton).
By differentiation, one checks that, for all $v \in \R$, $\uu(t,x)=u(x-vt)$ (where $u(x)=1-\exp\{-e^{-\la(x-a)}\}$) is a solution of 
\begin{equation} \label{eq:KPPG}
\uu_t = \uu_{xx} + A(\uu), 
\end{equation}
with reaction term
$$
 A(u)=\la (1-u) 
\big[ \la \ln \frac{1}{1-u} + (v - \la)\big]
\ln\frac{1}{1-u} .
$$
Since $A(0)=A(1)=0$,  the values  $u=0$ and $u=1$ are equilibria.
For $v \ge \la$, we have $A(u)>0$ for all $u \in (0,1)$, hence 
these values are the unique  equilibria  $u \in [0,1]$, with $u=0$ unstable
and $u=1$ stable.
For $v \in [\la, 3\la)$, $A$ is convex in the neighborhood of 0, so the  equation is not of KPP type
\cite[p.2]{GiKe}.

\subsection{Exponential tails:  front profile and traveling wave}\label{sec:expTW}

In this section we prove Theorem \ref{th:cvTW}. 
We consider the case of $\xi$ with exponential upper tails, $1-F(x)=\P(\xi > x) \sim e^{-x}$
as $x \to +\8$, that  can be written as
\begin{equation} \label{eq:eps}
\lim_{x \to +\8} \eps(x)=0, \qquad {\rm with} \quad \eps(x)=1+e^x \ln {F(x)} .
\end{equation}
 (By affine transformation, we also cover the case of tails
$\P(\xi > x) \sim e^{\la (x-a)}$.) By definition, $\eps(x) \in [-\8,1]$. 

We let $N \to \8$, keeping $t$ fixed and we use $\Phi$
from (\ref{eq:Phigumbel}) with $\la=1$. 
To show that the empirical distribution function (\ref{def:empiric}) converges, after the proper shift,
to that of the Gumbel distribution with the same tails, we will use the  stronger assumption
that
\begin{equation} \label{eq:eps2}
\lim_{x \to +\8} \eps(x)=0, \qquad {\rm and} \quad  \eps(x) \in [-\delta^{-1},1-\delta], 
\end{equation}
for all $x$ with some $\delta>0$.

\medskip

\begin{proof} (Theorem \ref{th:cvTW})  First of all, note that
$\ln F(x)= -(1-\eps(x)) e^{-x}.$
Let $m_i= e^{X_i(t-1)-  \Phi(X(t-1))}$, which add up to 1 by our choice of $\Phi$, and
let also $\eps_i= \eps (x+  \Phi(X(t-1))-X_i(t-1))$.

We start with the case $k=1, K_N=\{j\}$.
 From (\ref{eq:antoine}),
\begin{eqnarray} \nonumber
&&
\!\!\!\!\!\!\!\!\!\!\!\!\!\!\!\!\!\!\!\!\!\!\!\!\!\!\!\!\!\!\!\!\!\!\!\!\!\!\!\!\!
 \ln \P( X_j(t) - \Phi(X(t-1)) \leq x \vert \f_{t-1}) 
\qquad \qquad \qquad \qquad \qquad
\\ \nonumber
&=&
 \sum_{i=1}^N \ln F\big(x+  \Phi(X(t-1))-X_i(t-1)\big)\\ \nonumber
 &=&
 - \sum_{i=1}^N e^{-x-  \Phi(X(t-1))+X_i(t-1)} [1-\eps (x+  \Phi(X(t-1))-X_i(t-1))]  \\ \nonumber
 &=&
  - e^{-x} \sum_{i=1}^N m_i [1-\eps_i]  \\ \label{eq:ar1}
 &=&
  - e^{-x} \left(1 - \sum_{i \in I_1} m_i \eps_i - \sum_{i \in I_2} m_i \eps_i   \right)
   \end{eqnarray}
with $I_1=\{i: X_i(t-1) \leq \Phi(X(t-1)) - A\}$ and $I_2$ the complement in $\{1,\ldots N\}$,
and some real  number $A$ to be chosen later. By the first assumption in  (\ref{eq:eps2}),  we have
$$
|\sum_{i \in I_1} m_i \eps_i| \leq \sup\{ |\eps (y)|;  y > x+A\} \times 1 \to 0 \quad {\rm as} \; A \to \8,
$$
 for fixed $x$. The second sum, 
$$
|\sum_{i \in I_2}m_i \eps_i| \leq \| \eps \|_\8 \sum_{i \in I_2}  e^{X_i(t-1)-  \Phi(X(t-1))} ,
$$
will be bounded using the second assumption in (\ref{eq:eps2}). 
We can enlarge the probability space and couple the $\xi_{i,j}(s)$'s with 
$(g_{i,j}(t-1); i,j \leq N)$, which are  i.i.d. $G(0,1)$  independent of $(\xi_{i,j}(s); i,j \leq N, s \neq t-1)$, such that 
$$
g_{i,j}(t-1)+c \leq \xi_{i,j}(t-1) \leq g_{i,j}(t-1)+d.
$$
Define for $i \leq N$,
$$
\tilde X_i(t-1)= \max_{j \leq N} \{ X_j(t-2)+ g_{i,j}(t-1)\}.
$$
By the previous double inequality, 
$$
\tilde X_i(t-1) + c \leq X_i(t-1) \leq \tilde X_i(t-1) + d,
$$
and, since $\Phi$ is non-decreasing and such that $\Phi(y+r {\bf 1})
 = \Phi(y)+r $, we have also
$$
\Phi(X(t-1)) - \Phi( X(t-2)) \geq   \Phi(\tilde X(t-1))- \Phi(X(t-2))+c,  
$$
On the other hand, in analogy to the proof of Proposition \ref{g-case} for the Gumbel
case we know that $(\tilde X_i(t-1)-\Phi(X(t-2)); 
1 \leq i \leq N)$ is a $N$-sample of the law
$G(0,1)$. So, 
\begin{eqnarray*}
\Phi( \tilde X(t-1))- \Phi(X(t-2)) 
&=& \ln (b_N) + \ln \left( 1+ \frac{N}{b_N} \S^{(N)}\right)\\
&=& \ln N + \ln \ln N + o( 1) 
\end{eqnarray*}
in probability from (\ref{eq:cvstable}), and
$$ 
\max \{\tilde X_i(t-1); i \leq N\}- \Phi(X(t-2)) - \ln N \quad   {\rm converges\ in \ law}
$$
by the limit law  for the maximum of i.i.d.r.v.'s with exponential tails  \cite[Sect. I.6]{LeadbetterLR83}.
Combining these, we obtain, as $N \to \8$,
\begin{eqnarray*}
\Phi(  X(t-1))- \max \{ X_i(t-1); i \leq N\} &\geq &
\Phi( \tilde X(t-1))-d- \max \{\tilde X_i(t-1); i \leq N\} +c\\  
&\to&+\8 \quad {\rm \ in \ probability\ } , 
\end{eqnarray*}
which implies that the set $I_2$ becomes empty for fixed $A$ and increasing $N$.
 This shows that $ \sum_{i \in I_2}  e^{X_i(t-1)-  \Phi(X(t-1))}  \to 0$ in probability
(i.e., under $\P( \cdot | {\mathcal F}_{t-2})$)
uniformly on $X(t-2)$. Letting $N \to \8$ and $A \to +\8$ in (\ref{eq:ar1}), we have
$$
\P( X_j(t) - \Phi(X(t-1)) \leq x \vert \f_{t-2}) \to \exp - e^{-x}
$$
as $N \to \8$ uniformly on $X(t-2)$, which implies the first claim for $k=1$. For $k \geq 2$, 
recall that, conditionally on ${\mathcal F}_{t-1}$, the variables 
$(X_i(t); i \leq N)$ are independent. The previous arguments apply, yielding
(\ref{eq:cvTW1}).

Statement (\ref{eq:cvTW2}) for fixed $x$ follows from this and the fact that $X_i(t)$ are  independent conditionally on ${\mathcal F}_{t-1}$.
Convergence uniform for $x$ in compacts follows from pointwise convergence of monotone functions to a continuous limit
(Dini's theorem). Uniform convergence on $\R$ comes from the additional property that these functions are bounded by 1.
\end{proof}


\begin{rmk}
(i) From the stochastic comparison (\ref{eq:HIC}) of $\xi$ and the Gumbel, we obviously have
$v_N=\ln b_N + {\mathcal O}(1)$. We believe, but could not prove, that the error term
is in fact $o(1)$.
\\
(ii) 
 We believe, but could not prove, that the conclusions of 
Theorem \ref{th:cvTW} 
hold  under the only assumption that the function $\eps$ from (\ref{eq:eps}) tends to 0 at $+\8$.
\end{rmk}


\section{Front speed for the Bernoulli distribution}   \label{sec:bernoulli}

In this section we consider the case of a Bernoulli distribution
for the $\xi$'s,
$$
\P( \xi_{i,j}(t)=1)=p,\qquad \P( \xi_{i,j}(t)=0)=q=1-p,
$$
with $p \in (0,1)$.
For all starting configuration, from the coupling argument in the proof of 
proposition \ref{one}, we see that all
$N$ particles meet at a same location  at a geometric time, and, 
at all later times, they  share the location of the leading one, or they
lye at  a unit distance behind the leading one. We set 
$\Phi(x)=\max\{x_j; j \leq N\}$,  and we reduce the process $X^0$ to a simpler 
one given by considering
\begin{equation}
  \label{eq:defZ0}
Z(t)= \sharp \big\{j: 1 \leq j \leq N, X_j(t)=1+\max\{X_i(t-1); 
i \leq N\}\big\}.
\end{equation}
$Z(t)$ is equal to the number of leaders if the front has
moved one step forward at time $t$, and to 0 if the front stays
at the same location. Here, we define the front location as the rightmost occupied site $\Phi(X(t))=\max\{X_j(t); j \leq N\}$. 
Then, it is easy to see that 
$Z$ is a Markov chain on $\{0,1,\ldots, N\}$ with transitions
given by the binomial distributions 
\begin{equation}
  \label{eq:defZZ0}
\P\big( Z(t+1)= \cdot \;\vert Z(t)=m\big) 
= 
\left\{
\begin{array}{ll}
{\mathcal B} \big(N, 1-q^m\big)(\cdot), & m \geq 1,\\
{\mathcal B} \big(N, 1-q^N\big)(\cdot) , & m =0.
\end{array}
\right.
\end{equation}
Note that the chain has the same law on the finite set $\{1,2,\ldots\}$ when starting from 0 or from $N$.
Clearly, $v_N \to 1$ as $N \to \8$. We prove that the convergence is extremely 
fast.
\begin{teo} 
\label{th:bernoulli}
In the Bernoulli case, we have
\begin{equation}
  \label{eq:vbernoulli}
v_N = 1 - q^{N^2} 2^N + o(q^{N^2} 2^N)
\end{equation}
as $N \to \8$. 
\end{teo}
\begin{proof}
The visits at 0 of the chain $Z$ are the times when the front fails to move 
one step. Thus,
$$
\Phi(X(t))= \Phi(X(0))+ \sum_{s=1}^t {\bf 1}_{Z(s)\neq 0},
$$
which implies by dividing by $t$ and letting $t \to \8$, that
$$
v_N=\bar \nu_N(Z \neq 0)=1-\bar \nu_N(Z = 0),
$$
where $\bar \nu_N$ denotes the invariant (ergodic) distribution of the chain $Z$.
Let $E_N, P_N$ refer to the chain starting at $N$, and $T_k=\inf\{
t \geq 1: Z(t)=k\}$ the time of first visit at $k\; (0 \leq k \leq N)$. 
By Kac's lemma, we can express the invariant distribution, and get:
\begin{equation}
  \label{eq:v_NT_0}
v_N = 1 - (E_0 T_0)^{-1}= 1 - (E_N T_0)^{-1}.
\end{equation}
Let $\sigma_0=0$, and $\sigma_1, \sigma_2\ldots$ the successive passage times
of $Z$ at $N$, and 
${\mathcal N}=\sum_{i \geq 0} {\bf 1}_{\sigma_i < T_0}$ the number of 
visits at $N$ before hitting $0$. 
Note that ${\mathcal N}$ has a geometric law with 
expectation $E_N {\mathcal N}=P_N(T_0 < T_N)^{-1}$.
Then,
\begin{eqnarray}  \nonumber
E_N T_0 &=&
E_N \left[\sum_{i \geq 1} (\sigma_{i}-\sigma_{i-1}) {\bf 1}_{\sigma_i < T_0}
+ (T_0 - \sigma_{\mathcal N}) \right]\\  \nonumber
 &=&
\sum_{i \geq 1}
E_N \left[ (\sigma_{i}-\sigma_{i-1}) {\bf 1}_{\sigma_i < T_0} \right]
+ E_N (T_0 - \sigma_{\mathcal N}) 
\\   \nonumber
 &=&
\sum_{i \geq 1}
E_N \left[ {\bf 1}_{\sigma_{i-1} < T_0} E_N \big( 
 \sigma_1 {\bf 1}_{\sigma_1 < T_0}\big)\right]
+  E_N (T_0 \vert T_0 < T_N) \qquad {\rm (Markov\ property)}
\\  \nonumber
 &=&
E_N [ {\mathcal N}]
\times
E_N \big( 
 \sigma_1 {\bf 1}_{\sigma_1 < T_0}\big)
+  E_N (T_0 \vert T_0 < T_N) \\
&=& \frac{1-P_N(T_0 < T_N) 
}{P_N(T_0 < T_N) }
\times
E_N \big( 
 T_N \vert {T_N < T_0}\big)
+  E_N (T_0 \vert T_0 < T_N) \label{eq:xyz}
\end{eqnarray}
We will prove a Lemma.

\begin{lem} \label{lem:tps}
We have
 \begin{equation}
  \label{eq:tps1}
P_N(T_0 < T_N) \sim q^{N^2} 2^N,
\end{equation}
as $N$ tends to $\8$. Moreover,
 \begin{equation}
  \label{eq:tps2}
\lim_{N \to \8}  E_N (T_0 \vert T_0 < T_N)= 2,
\end{equation}
 \begin{equation}
  \label{eq:tps3}
\lim_{N \to \8}
E_N \big( 
T_N \vert {T_N < T_0}
\big)
=1.
\end{equation}
\end{lem}
The lemma has a flavor of Markov chains with rare transitions considered in
\cite{CaCe95/97}, except for the state space which is getting here
larger and 
larger in the asymptotics.
With the lemma at hand,  we conclude that
\begin{eqnarray*}
E_N T_0 &\sim& \frac{1}{P_N(T_0 < T_N)} \\
&\sim& \frac{1}{q^{N^2} 2^N}
\end{eqnarray*}
as $N$ tends to $\8$. From (\ref{eq:v_NT_0}), this implies the
statement of the theorem.
\end{proof}


\begin{proof} of lemma \ref{lem:tps}.
We start to prove the key relation (\ref{eq:tps1}).
We decompose the event $\{T_0 < T_N\}$ according to the number
$\ell$ of steps to reach 0 from state $N$,
 \begin{equation}
  \label{eq:obv1}
P_N(T_0 < T_N) = \sum_{\ell \geq 1}
P_N(T_0=\ell < T_N).
\end{equation}
We directly compute
the contribution of $\ell=1$:
By (\ref{eq:defZ0}), we have
 \begin{equation}
  \label{eq:obv2}
P_N(T_0=1 < T_N)=q^{N^2},
\end{equation}
which is neglegible in front the right-hand side of  (\ref{eq:tps1}).
We compute now the contribution of strategies in two steps:
\begin{eqnarray}  \nonumber
P_N(T_0=2 < T_N) 
&=&
\sum_{k=1}^{N-1}
P_N(T_0=2 < T_N, Z(1)=k) \\ \nonumber
&=&
\sum_{k=1}^{N-1}
\begin{pmatrix}
N\\k
\end{pmatrix}
(1-q^N)^k q^{N(N-k)} \times \begin{pmatrix}
N\\0
\end{pmatrix}
(1-q^k)^0 q^{kN} \\ \nonumber
&=& q^{N^2}
\sum_{k=1}^{N-1}
\begin{pmatrix}
N\\k
\end{pmatrix}
(1-q^N)^k \\  \nonumber
&=&  q^{N^2} \big[ (2-q^N)^N-1-(1-q^N)^N\big]
\\
  \label{eq:obv3}
&\sim& q^{N^2} 2^N.
\end{eqnarray}

For $\ell \geq 2$ we write, with the convention that $k_0=N$,
\begin{eqnarray}  \nonumber
P_N(T_0=\ell+1 < T_N)
&=&
\sum_{k_1, \ldots k_\ell=1}^{N-1}
P_N(T_0=\ell+1 < T_N, Z(i)=k_i, i=1,\ldots \ell) \\ \nonumber
&=&
\sum_{k_1, \ldots k_\ell=1}^{N-1} \Big[ \prod_{i=1}^\ell
\begin{pmatrix}
N\\k_i
\end{pmatrix}
(1-q^{k_{i-1}})^{k_i} q^{k_{i-1}(N-k_i)} \Big]
q^{k_\ell N} 
\qquad ({\rm by\ } (\ref{eq:defZ0}))
\\ \nonumber
&\leq&
\sum_{k_1, \ldots k_\ell=1}^{N-1} \Big[ \prod_{i=1}^\ell
\begin{pmatrix}
N\\k_i
\end{pmatrix}
 q^{k_{i-1}(N-k_i)} \Big]
q^{k_\ell N} \\ \nonumber
&=&
q^{N^2}
\sum_{k_1, \ldots k_\ell=1}^{N-1} \prod_{i=1}^\ell
\begin{pmatrix}
N\\k_i
\end{pmatrix}
q^{k_{i}(N-k_{i-1})} \qquad \qquad \qquad{({\rm since \ }k_0=N)}
 \\ \label{eq:prberrec}
&=:& q^{N^2} a_\ell,
\end{eqnarray}
which serves as definition of $a_\ell=a_\ell(N)$. For $\eps \in (0,1)$, 
define also
$b_\ell= b_\ell(\eps, N)$ by
$$
b_\ell=
\sum_{
1 \leq k_1, \ldots k_\ell \leq N-1,
k_\ell > (1-\eps)N
}
\;
\prod_{i=1}^\ell
\begin{pmatrix}
N\\k_i
\end{pmatrix}
q^{k_{i}(N-k_{i-1})} 
$$
Then, by summing over $k_\ell$,
\begin{eqnarray}  \nonumber
a_\ell
&=&
\sum_{1 \leq k_1, \ldots k_{\ell -1} \leq N-1}
\Big[ \prod_{i=1}^{\ell -1}
\begin{pmatrix}
N\\k_i
\end{pmatrix}
 q^{k_{i}(N-k_{i-1})} \Big]
 \big[ (1+q^{N-k_{\ell-1}})^N -1 - q^{N(N-k_{\ell-1})}\big]
\\ \nonumber
&\leq&
\sum_{1 \leq k_1, \ldots k_{\ell -1} \leq N-1}
\Big[ \prod_{i=1}^{\ell -1}
\begin{pmatrix}
N\\k_i
\end{pmatrix}
 q^{k_{i}(N-k_{i-1})} \Big]
 \big[ (1+q^{N-k_{\ell-1}})^N -1\big]\\   \nonumber
&=& 
\sum_{ k_{\ell -1} \leq (1-\eps)N} \qquad
+ \qquad
\sum_{ k_{\ell -1} > (1-\eps)N}
\\
\label{eq:344-5}
&\leq&
\gamma_N \; a_{\ell-1} + (1+q)^N b_{\ell-1},
\end{eqnarray}
with
$$
\gamma_N=\gamma_N(\eps) \stackrel{\rm def.}{=} \left(1+q^{N\eps}\right)^N -1 
\sim N q^{N\eps}
$$
as $N \to \8$. 
We now bound $b_\ell$ in a similarly manner. First we note that, for any 
$\eta$ such that
$\eta > -\eps \ln ( \eps) - (1-\eps) \ln  (1-\eps)>0$, we have 
$$\sum_{k_\ell > (1-\eps)N} \begin{pmatrix}
N\\k_\ell \end{pmatrix} \leq \exp N \eta \qquad {\rm for\ large\ }N. $$
Note also that we can make $\eta$ small by choosing $\eps$ small.
Then,
\begin{eqnarray}  \nonumber
b_\ell
&\leq &
\sum_{1 \leq k_1, \ldots k_{\ell-1} \leq N-1}\;  \left[
\prod_{i=1}^{\ell-1}
\begin{pmatrix}
N\\k_i
\end{pmatrix}
q^{k_{i}(N-k_{i-1})} \right]
e^{ N \eta}
q^{(1-\eps)N(N-k_{\ell-1})} \\  \nonumber
&=& \sum_{ k_{\ell-1} \leq (1-\eps)N} \qquad
+ \qquad
\sum_{ k_{\ell-1} > (1-\eps)N}
\\ \label{eq:novo}
&\leq &
q^{\eps(1-\eps)N^2} e^{ N \eta} a_{\ell-1} + 
q^{(1-\eps)N} e^{ N \eta} b_{\ell-1}.
\end{eqnarray}
For vectors $u, v$, we write  $u \leq v$ if the inequality holds 
coordinatewise. In view of (\ref{eq:344-5}) and (\ref{eq:novo}),
we finally have
\begin{equation} \label{eq:prodmatr}
\begin{pmatrix}
a_\ell\\b_\ell
\end{pmatrix}
\leq 
M 
\begin{pmatrix}
a_{\ell-1}\\b_{\ell-1}
\end{pmatrix}
\leq \ldots \leq 
M^{\ell -2} 
\begin{pmatrix}
a_{2}\\b_{2}
\end{pmatrix},
\end{equation}
where the matrix $M$ is positive and given by 
$$M=\begin{pmatrix}
\gamma_N & (1+q)^N 
\\
q^{\eps(1-\eps)N^2} e^{ N \eta} &
q^{(1-\eps)N} e^{ N \eta} 
\end{pmatrix}$$
We easily check that, for $\eps$ and $\eta$ small, $M$ has positive, real 
eigenvalues, and
the largest one $\la_+=\la_+(N, \eps, \eta)$
is such that $\la_+ \sim \gamma_N$ as $N \to \8$. By (\ref{eq:prodmatr}),
$$
a_\ell \leq \la_+^{\ell -2} \big( a_2 + b_2),
$$
and since $\la_+= \la_+(N,\eps,\eta)<1$ for large $N$, 
\begin{equation} \label{eq:serie}
\sum_{\ell \geq 2} a_\ell \leq \frac{a_2+ b_2}{1 -\la_+}.
\end{equation}
Now, we estimate $a_2=a_2$ and $b_2$, both of which depend on $N$:
\begin{eqnarray*}
a_2 &=& 
\sum_{1 \leq k_1, k_2 \leq N-1}
\begin{pmatrix}
N\\k_1
\end{pmatrix}
\begin{pmatrix}
N\\k_2
\end{pmatrix}
 q^{k_{2}(N-k_{1})} \\
 &\leq& 
\sum_{1 \leq k_1 \leq N-1} \begin{pmatrix}
N\\k_1
\end{pmatrix}
\big[ (1+q^{N-k_1})^N-1\big]\\
 &=& \sum_{ k_1 \leq (1-\eps)N} \qquad
+ \qquad
\sum_{ k_1 > (1-\eps)N}\\
 &\leq& 
\gamma_N 2^N + (1+q)^N e^{N \eta},
\end{eqnarray*}
and 
$$b_2 \leq  
q^{(1-\eps)N} \sum_{1 \leq k_1, k_2 \leq N-1, k_2 > (1-\eps) N}
\begin{pmatrix}
N\\k_1
\end{pmatrix}
\begin{pmatrix}
N\\k_2
\end{pmatrix}
\leq
2^N q^{(1-\eps)N}e^{N \eta}.
$$
From (\ref{eq:serie}), we see that
$$
\sum_{\ell \geq 2} a_\ell(N) = o( 2^N),
$$
and, together with (\ref{eq:prberrec}), (\ref{eq:obv1}), (\ref{eq:obv2}),
(\ref{eq:obv3}), it implies   (\ref{eq:tps1}). 
\medskip

The limit (\ref{eq:tps2}) directly follows from the above estimates. 

\medskip

Finally, we turn to the proof of (\ref{eq:tps3}). 
Note that
\begin{equation}
\label{eq:s2-1} 
P_N ({T_N < T_0}) \leq  E_N \big( 
 T_N  {\bf 1}_{ T_N < T_0}\big) \leq 
 E_N \big( T_N \big) =
 \Big(1-(1-p)^N\Big)^N +  E_N \big( 
 T_N  {\bf 1}_{T_N  >1}\big) \;.
\end{equation}
We only need to show that the last term is exponentially small.
For that, we use Markov property at time 1,
$$
E_N \big( 
 T_N  {\bf 1}_{T_N  >1}\big) \leq 
 \Big[1- \Big(1-(1-p)^N\Big)^N \Big] \Big( 1+ \max_m E_m(T_N )\Big),
$$
where the first factor is exponentially small.  To show that the second factor is bounded, 
one can repeat the proof of  part a) with $p=1$ of the forthcoming  Lemma \ref{bound-moment}.
\end{proof}

\section{The case of variables taking a countable
number of values}   \label{sec:trinomial}

In this section we consider the case of a random variable $\xi$
taking the values $\mathbb N_k:=\{l\in\mathbb Z:l\le k\}$, with $k\in {\mathbb Z}$, so that
\begin{equation}
\label{xi}
\P( \xi_{i,j}(t)=l)=p_l,
\end{equation}
for $l\in\mathbb N_k$,
with $p_l \geq 0, p_k \in (0,1)$ and $\sum_{l\in\mathbb N_k} p_l=1$.
As in the Bernoulli case we can reduce the process $X^0$ to a simpler 
one given by $Z(t):=(Z_l(t):l\in\mathbb N_k )$, where
\begin{equation}
  \label{eq:defZ}
Z_l(t)= \sharp \big\{j: 1 \leq j \leq N, X_j(t)=\max\{X_i(t-1); 
1\le i \leq N\} + l \big\},
\end{equation}
for $l\in \mathbb N_k$. Note that $Z_k(t)$  is equal to the number of leaders if the front has
moved $k$ steps forward at time $t$, and to 0 if the front 
moved less than $k$ steps.
$Z$ is a Markov chain on the set 

$$
\Omega_k:=
\left\{m\in \{0,\ldots,N\}^{\mathbb N_k}:\sum_{i\in\mathbb N_k} m_i=N\right\},
$$
where $m_i$ are the coordinates of $m$.
We now proceed to compute the
 transition
probabilities of the Markov chain $Z$.
Assume that at some time $t$ we have
$Z_t=m=(m_i:i\in\mathbb N_k)$. For each $i\in\mathbb N_k$,  
this corresponds to $m_i$ particles
at position $i$. Let us now move each particle to the right adding
independently a random variable with law $\xi_{0,0}$.
We will assume that $m_k\ge 1$.
The probability that  at time $t+1$ there is some particle
at position $k$ 
 is
$$
s_k(m):=1-\left(\sum_{l=-\infty}^{k-1}p_l\right)^{m_k}.
$$
Similarly, the probability that the rightmost particle
at time $t+1$ is at position $k-1$ is
$$
s_{k-1}(m):=\left(\sum_{l=-\infty}^{k-1}p_l\right)^{m_k}-
\left(\sum_{l=-\infty}^{k-1}p_l\right)^{m_{k-1}}
\left(\sum_{l=-\infty}^{k-2}p_l\right)^{m_k}.
$$
In general, for $r\in\mathbb N_k$, the
probability that at time $t+1$ the rightmost
particle is at position $r$ is

\begin{eqnarray}
\nonumber
s_r(m)&:=&
\left(\sum_{l=-\infty}^{k-1}p_l\right)^{m_{r+1}}
\cdots
\left(\sum_{l=-\infty}^{r}p_l\right)^{m_{k}}\\
\nonumber
&&\qquad -
\left(\sum_{l=-\infty}^{k-1}p_l\right)^{m_{r}}
\cdots
\left(\sum_{l=-\infty}^{r}p_l\right)^{m_{k-1}}
\left(\sum_{l=-\infty}^{r-1}p_l\right)^{m_{k}}.
\end{eqnarray}
Define now on $\Omega_k$
the shift $\theta m$ by $(\theta m)_i=m_{i-1}$ for $i\in\mathbb N_k$.
For $r\in\mathbb N_k$, let
$s^{(1)}_r(m):=s_r(\theta m)$ and in general for $j\ge 1$ let

$$
s^{(j)}_r(m):=s_r(\theta^{j}m).
$$
Define $s(m):=(s_r(m):r\in\mathbb N_k)$, for $j\ge 1$,
$s^{(j)}(m):=(s^{(j)}_r(m):r\in\mathbb N_k)$.
Dropping the dependence on $m$ of $s_r$, $s^{(j)}_r$, $s$ and $s^{(j)}$ we can
now write the transition probabilities of
the process $Z(t)$ as

\begin{eqnarray}
\nonumber
&\!\!\!\!\!\!\!\!\!\!\!\!\!\!\!\!\!\!\!\!\!\!\!\!\!\!\!\!\!\!\!\!\!\!\!\!\!\!\!\!\!\!\!\!\!\!\!\!\!\!\!\!\!\!\!\!\!\!\!\!\!\!\!\!\!\!\!\!\!\!\!\!
\!\!\!\!\!\!\!\!\!\!\!\!\!\!\!\!\!\!\!\!\!\!\!\!\!\!\!\!\!\!\!\!\!\!\!\!\!\!\!\!\!\!\!\!\!\!\!\!\!\!\!\!\!\!\!\!\!\!\!
\P\big( Z(t+1)= n\vert Z(t)=m\big) \qquad\qquad
 \qquad\qquad\qquad\qquad\\ 
 \qquad\qquad
  \label{eq:defQ}
&  \qquad= \qquad
\left\{
\begin{array}{ll}
{\mathcal M} \big(N; s\big)(n), & m_k \geq 1,\\
{\mathcal M} \big(N; s^{(1)}\big)(n), & m_{k-1}\ge 1, m_k
=0,\\
{\mathcal M} \big(N; s^{(2)}\big)(n), & m_{k-2}\ge 1, m_k=m_{k-1} =0,\\
\ldots & \ldots\\ 
{\mathcal M} \big(N; s^{(j)}\big)(n), &  m_{k-j}\ge 1,m_{k}=m_{k-1}=\cdots=m_{k-j+1}=0,
\end{array}
\right.
\end{eqnarray}
where for $u_i$ with $\sum_i u_i=1$, ${\mathcal M} \big(N; u \big)$ denotes the multinomial distribution (with infinitely many classes).
Let us introduce the following notation

$$
r_i:=\sum_{j=-\infty}^{k-i}p_j,
$$
for integer $i\ge 1$.

\smallskip

\noindent {\bf Assumption (R)}. We say that a random
variable $\xi$ distributed according to (\ref{xi}) 
satisfies assumption (R) if 

$$
p_k \times p_{k-1}>0
$$
and
$$
\mathbb E(|\xi_{0,0}|)<\infty.
$$

\smallskip

\noindent We can now state the main result of this section.

\begin{teo} 
\label{th:multinomial}
Let $\xi$ be distributed according to
(\ref{xi}) and suppose that it satisfies
assumption  (R). Then, we have that
\begin{equation}
  \label{eq:multinomial}
v_N = k - q_k^{N^2} 2^N + o(q_k^{N^2} 2^N),
\end{equation}
as $N \to \8$, 
where $q_k:=1-p_k$.
\end{teo}


\subsection{Proof of Theorem \ref{th:multinomial}}
To prove  Theorem \ref{th:multinomial}, we will follow a strategy similar
to the one used in the Bernoulli case. Let us first define for each
$m=(m_i:i\in\mathbb N_k)\in\Omega_k$ the function

\begin{equation}
\label{phi}
\phi=\phi(m):=\sup\{i\in\mathbb N_k:m_i>0\}.
\end{equation}
As in the Bernoulli case, we denote by $\bar\nu_N$
the invariant (ergodic) distribution of the chain $Z$.
\begin{lem}
\label{veloc}
Let $\xi$ be distributed according to
(\ref{xi}). Then, we have that

\begin{equation}
\label{veloc2}
v_N=k-\bar\nu_N(\phi\le k-1)-\sum_{j=2}^\infty \bar\nu_N(\phi\le k-j).
\end{equation}
\end{lem}
\begin{proof}
Let $\Phi(x)=\max\{x_i; i \leq N\}$, and  note that for every positive integer time  $t$

$$
\Phi(X(t))=\Phi(X(0))+\sum_{i=1}^t
\phi(Z(i)).
$$
Hence
\begin{eqnarray}
v_N  = \sum_{i \in \N_k}i\bar\nu_N(\phi=i) =
\label{expansion}
 k-\sum_{j=1}^\infty \bar\nu_N(\phi\le k-j).
\end{eqnarray}
\end{proof}

\noindent
We will now show that the first two terms of the expression
for the velocity (\ref{veloc2}) given in Lemma \ref{veloc}, dominate the others.
\begin{lem}
\label{l2}  Let $\xi$ be distributed according to
(\ref{xi}). Then, for each $i\ge 2$ we have that

$$
\bar\nu_N(\phi\le k-i)\le \left(\frac{r_i}{r_1}\right)^N\bar\nu_N(\phi\le k-1).
$$

\end{lem}
\begin{proof}  Let us fix $m\in\Omega_k$. 
Define $\kappa:=\sup\{i\in\mathbb N_k:m_i>0\}$.
Let us first note that

$$
P_{m}(\phi(Z(1))\le k-1)=
r_1^{m_\kappa N},
$$
while

$$
P_{m}(\phi(Z(1))\le k-2)=
r_1^{m_{\kappa-1}N}r_2^{m_\kappa N}\le\left(\frac{r_2}{r_1}\right)^{m_\kappa N}
P_{m}(\phi(Z(1))\le k-1).
$$
Hence

\begin{eqnarray*}
P_m(\phi(Z(t))\le k-2)&=&
\sum_{m'\in\Omega_k}P_m(Z(t-1)=m',\phi(Z(t))\le k-2)\\
&=&
\sum_{m'\in\Omega_k}P_m(Z(t-1)=m')P_{m'}(\phi(Z(1))\le k-2)\\
&\le&
\sum_{m'\in\Omega_k}P_m(Z(t-1)=m')P_{m'}(\phi(Z(1))\le k-1)\left(\frac{r_2}{r_1}\right)^{m'_{\kappa} N}\\
&
\le&
\left(\frac{r_2}{r_1}\right)^{ N}
P_m(\phi(Z(t))\le k-1),
\end{eqnarray*}

where in the last inequality we used the fact that by definition 
$m'_{\kappa'}\ge 1$. A similar reasoning shows that in general, 
for $i\ge 2$,

$$
P_m(\phi(Z(t))\le k-i)\le
\left(\frac{r_i}{r_1}\right)^{ N}P_m(\phi(Z(t))\le k-1).
$$
Taking the limit when $t\to\infty$ and using Proposition \ref{one},
we conclude the proof.
\end{proof}

\begin{lem} 
\label{on}
Let $\xi$ be distributed according to
(\ref{xi}) and suppose that assumption (R) is satisfied. Then

$$
\sum_{i=2}^\infty \left(\frac{r_i}{r_1}\right)^N={\mathcal O} \left(  \left(\frac{r_2}{r_1}\right)^N\right)
$$
\end{lem}

\begin{proof}
Note that by summation by parts, assumption (R)
implies that

$$
\sum_{i=2}^\infty r_i<\infty.
$$
Therefore, 

$$
\sum_{i=2}^\infty \left(\frac{r_i}{r_1}\right)^N
\le \frac{1}{r_1}\left(\frac{r_2}{r_1}\right)^{N-1}\sum_{i=2}^\infty r_i
={\mathcal O} \left(  \left(\frac{r_2}{r_1}\right)^N\right).
$$
\end{proof}

\noindent Theorem \ref{th:multinomial} now follows from Lemmas \ref{veloc},
\ref{l2}, \ref{on} and the next proposition, whose proof we defer to
subsection \ref{proplim}. 

\begin{prop}
\label{propo} We have that

$$
\lim_{N\to\infty}\frac{\bar\nu_N(\phi\le k-1)}{q_k^{N^2}2^N}=1.
$$

\end{prop}

\subsection{Proof of Proposition \ref{propo}}
\label{proplim}
Let us introduce for each $m\in\Omega_k$
the stopping time

$$
T_m:=\inf\{t\ge 1: Z(t)=m\}.
$$
Define now $\Omega_k^0:=\{m\in\Omega_k: m_k=0\}$.
Furthermore, we denote in this section
$\oplus:=(\ldots,0,N)\in\Omega_k$.
We now note that by Kac's formula

\begin{equation}
\nonumber
\bar\nu_N(Z_k=0)=\sum_{n\in\Omega_k^0} \bar
\nu_N(Z=n)
=
\sum_{n\in\Omega_k^0}\frac{1}{E_{n}(T_n)}.
\end{equation}
Hence we have to show that

\begin{equation}
\label{ee1}
\lim_{N\to\infty}
\frac{\sum_{n\in\Omega_k^0}\frac{1}{E_{n}(T_n)}}
{q_k^{N^2}2^N}=1.
\end{equation}
We will prove (\ref{ee1}) through the following three lemmas.

\begin{lem}
\label{ll1} Assume that $\xi$ is distributed according to 
 (\ref{xi}).
Then, for every $n\in\Omega_k^0$  we have that

\begin{equation*}
E_{n}(T_n)
=
E_\oplus(T_\oplus,
T_\oplus<T_n)
\frac{1}{P_\oplus(T_n<
T_\oplus)}+
E_\oplus(T_n|
T_n<T_\oplus)+U_N(n),
\end{equation*}
where $1-e^{-CN} \leq \inf_{n_1,\ldots,n_{k-1},0}|U_N| \le\sup_{n_1,\ldots,n_{k-1},0}|U_N|\le 2+e^{-CN}$ for some constant $C>0$.

\end{lem}

\begin{lem}
\label{bound-moment}
Assume that $\xi$ is distributed according to 
 (\ref{xi}).  Then, there is a constant $C>0$ such that the
following are satisfied.

\begin{itemize}

\item[a)] For $p=1$ and $p=2$, and for every $N\ge 2$ we have that

\begin{equation}
\label{afirst}
\sup_{m\in\Omega_k}E_m(T^p_\oplus)\le
2^p(1+ e^{-CN}).
\end{equation}

\item[b)] For every $N\ge 2$ we have that

$$
\sup_{m\in\Omega_k^0}\left| E_\oplus(T_\oplus,
T_\oplus<T_m)-1\right|\le
e^{-CN}.
$$

\end{itemize}
\end{lem}
To state the third lemma, we  need to define the first 
hitting time of the set $\Omega_k^0$. We let

$$
T_A:=\inf_{m\in\Omega_k^0}T_m.
$$
\begin{lem}
\label{red-bern}
Assume that $\xi$ is distributed according to 
 (\ref{xi}).  Then, there is a constant $C>0$ such that

$$
\sum_{n\in\Omega_k^0}P_\oplus(T_n<T_\oplus)
=P_\oplus(T_A<T_\oplus)\big(1+ {\mathcal O}(e^{-CN})\big).
$$

\end{lem}

\noindent Let us now see how Lemmas \ref{ll1}, \ref{bound-moment} and
\ref{red-bern}
imply  Proposition \ref{propo}. We will see
that in fact, Proposition \ref{propo} will follow
as a corollary  of the corresponding result for
 the Bernoulli case with $q=q_k$. Note that
Lemma \ref{ll1} and part $(b)$ of
Lemma \ref{bound-moment} imply that

$$
P_\oplus(T_n <
T_\oplus)
\ge
\frac{1-e^{-CN}}{E_n(T_n)
}, \qquad n \in \Omega_k^0.
$$
Hence, summing up over $n\in\Omega_k^0$, by Lemma \ref{red-bern}, we get that, for some $C'>0$,

\begin{equation}
\label{pup}
P_\oplus(T_A<T_\oplus)\ge
(1-e^{-C'N})
\sum_{n\in\Omega_k^0}
\frac{1}{E_n(T_n)}.
\end{equation}
Now, note that $P_\oplus(T_A<T_\oplus)$
is equal to the probability to hit $0$ before $N$,
starting from $N$, for the chain $Z$ defined through
random variables with Bernoulli increments as in Section 5.
Hence, by (\ref{eq:tps1}) of Lemma \ref{lem:tps}
we conclude that for $N$ large enough

\begin{equation}
\label{pbound}
(1+e^{-CN})q_k^{N^2}2^N\ge 
\sum_{n\in\Omega_k^0}
\frac{1}{E_{n}(T_{n})}.
\end{equation}
On the other hand, applying the Cauchy-Schwarz inequality
to the expectation $E_\oplus(\cdot |T_n<T_\oplus)$ in
 Lemma \ref{ll1} and using
Lemma \ref{bound-moment}, we obtain 
for each $n\in\Omega_k^0$ that

\begin{equation*}
E\le \frac{a_1}{P}+\frac{a_2}{\sqrt{P}}+a_3,
\end{equation*}
where $a_1:=1+e^{-CN}$, $a_2:=2(1+e^{-CN})$ and
$a_3:=U_N$, $E:=E_n(T_n)$,
$P:=P_\oplus(T_n<
T_\oplus)$
and we have used (\ref{afirst}) of part $(a)$ of Lemma \ref{bound-moment}
with $p=2$. It follows that

$$
\frac{1}{\sqrt{P}}
\ge \frac{\sqrt{a_2^2-4a_1(a_3-E)}-a_2}{2a_1}.
$$
Hence,

$$
a_1\frac{1}{P}\ge E
-\frac{a_2}{2a_1}\sqrt{a_2^2-4a_1(a_3-E)}.
$$
Now, $a_2^2-4a_1(a_3-E)\le 8(1+E)$ for large $N$, so that

$$
a_1 \frac{1}{P}\ge E\left(1
-4\frac{1}{\sqrt{E}}\sqrt{8\left(\frac{1}{E}+1\right)}\right).
$$
Now, by inequality (\ref{pbound}) we
conclude that for $N$ large enough $\frac{1}{E}\le q_k^{N^2}2^{N+1}$.
Therefore,

$$
\frac{1}{E_n
(T_n)}\ge 
(1-e^{-C'N})P_\oplus
(T_n<T_\oplus
).
$$
Summing up over $n\in\Omega_k^0$, by Lemma \ref{red-bern}
we get that

\begin{equation}
\label{pdown}
P_\oplus(T_A<T_\oplus)\le
(1+e^{-C'N})
\sum_{n\in\Omega_k^0}
\frac{1}{E_n(T_n)}
\end{equation}
for some $C'>0$. 
Finally, (\ref{eq:tps1}) of Lemma \ref{lem:tps},
together with inequalities (\ref{pup}) and (\ref{pdown}),
imply inequality (\ref{ee1}), which finishes the proof of Proposition \ref{propo}. \qed

\subsubsection{ Proof of Lemma \ref{bound-moment}}

\noindent {\it Part (a)}. We will first
prove that there exists a constant $C>0$ such that

\begin{equation}
\label{upb}
\sup_{m\in\Omega_k}P_m(T_\oplus>2)\le e^{-CN}.
\end{equation}
The strategy to prove this bound
will be to show that with a high probability, after
one step there are at least $\frac{p_kN}{2}$ leaders. This gives
a high probability of then having $N$ leaders in the second step.
Consider now the set $L_{k,N}:=\left\{m\in\Omega_k: m_k \geq \left[\frac{p_kN}{2}
\right]\right\}$.
We have

\begin{equation}
\label{fe1}
P_m(T_\oplus\le 2)\ge
P\left(X\ge \frac{p_k N}{2}\right)
\inf_{m\in L_{k,N}}P_m(T_\oplus=1),
\end{equation}
where $X$ is a random variable with a binomial distribution
of parameters $p_k$ and $N$. Now, by a large deviation
estimate, the first factor of (\ref{fe1}) is bounded from
below by $1-e^{-CN}$. On the other hand, we have
for $m\in L_{k,N}$,

$$
P_m(T_\oplus=1)\ge
\left(1-(1-p_k)^{Np_k/2}\right)^N\ge 1-e^{-CN},
$$
for some constant $C>0$. This estimate combined with
(\ref{fe1}) proves inequality (\ref{upb}).
Now, by the Markov property, we get that, for all $m\in\Omega_k$,

\begin{eqnarray*}
\nonumber
E_m(T_\oplus)
&=& E_m(T_\oplus  1_{T_\oplus \leq 2})
+ \sum_{n \in \Omega_k} E_m(T_\oplus  1_{T_\oplus > 2, Z(2)=n})\\
&\le& 2 P_m({T_\oplus \leq 2}) 
+ \sum_{n \in \Omega_k} E_m\big(1_{T_\oplus > 2, Z(2)=n}
[2+E_n(T_\oplus)]
\big)\\
&\le& 2  P_m(T_\oplus\le 2)
+\left(2+\sup_{n\in\Omega_k}
E_{n}
(T_{\oplus})\right) P_{m}(T_{\oplus}> 2),
\end{eqnarray*}
where the supremum is finite, in fact smaller than $\delta_N^{-1}$ with $\delta_N$ from
(\ref{deltaN}).
Bounding the first term of the right-hand side
of the above inequality by $2$, taking the supremum over $m \in \Omega_k$  and
applying the bound (\ref{upb}), we
obtain (\ref{afirst}) of $(a)$ of Lemma \ref{bound-moment}
with $p=1$. The proof of (\ref{afirst}) when $p=2$ 
is analogous via an application of the case $p=1$.

\smallskip

\noindent {\it Part (b)}. Note that
for every state $m\in\Omega_k^0$ we have that

$$
E_\oplus(T_\oplus,
T_\oplus<T_A)\le
E_\oplus(T_\oplus,
T_\oplus<T_m)\le
E_\oplus(T_\oplus).
$$
Hence, it is enough to prove that

\begin{equation}
\label{etn}
\left|E_\oplus(T_\oplus)-1\right|\le e^{-CN},
\end{equation}
and that
\begin{equation}
\label{etn2}
\left|E_\oplus(T_\oplus,T_\oplus<T_A)-1\right|\le e^{-CN}.
\end{equation}
To prove (\ref{etn}) note that

\begin{equation}
\label{int1}
E_\oplus(T_\oplus)=
\left(1-(1-p_k)^N\right)^N+E_\oplus(T_\oplus,T_\oplus>1).
\end{equation}
But by the  Markov property,

$$
E_\oplus(T_\oplus,T_\oplus>1)
\le \left(1-\left(1-(1-p_k)^N\right)^N\right)
\left(1+\sup_{m\in\Omega_k^0}E_m(T_\oplus)\right).
$$
Note that

$$
\left(1-(1-p_k)^N\right)^N
\ge \exp\left\{-\frac{N(1-p_k)^N}{1-N(1-p_k)^N}\right\}
\ge 1-\frac{N(1-p_k)^N}{1-N(1-p_k)^N}.
$$
Using part $(a)$ just proven  of this Lemma, we conclude that

\begin{equation}
\label{int3}
E_\oplus(T_\oplus,T_\oplus>1)
\le e^{-CN}.
\end{equation}
Substituting this back into (\ref{int1}) we
obtain inequality (\ref{etn}).
To prove inequality (\ref{etn2}), as before, observe that

\begin{equation}
\label{int2}
E_\oplus(T_\oplus, T_\oplus<T_A)=
\left(1-(1-p_k)^N\right)^N+E_\oplus(T_\oplus,T_A>T_\oplus>1).
\end{equation}
Noting that $E_\oplus(T_\oplus,T_A>T_\oplus>1)
\le E_\oplus(T_\oplus,T_\oplus>1)$, we
can use the estimate (\ref{int3}) to obtain (\ref{etn2}).

\subsubsection{Proof of Lemma \ref{ll1}}
We will use the following relation, which proof is similar to that of (\ref{eq:xyz}) and will be not be repeated here:
for every $n\in\Omega_k^0$,
\begin{equation}
\label{eq:xyzt}
E_\oplus(T_n)
=
E_\oplus(T_\oplus|
T_\oplus<T_n)
\frac{P_\oplus(
T_\oplus<T_n)
}{P_\oplus(T_n<
T_\oplus)}+
E_\oplus(T_n|
T_n<T_\oplus).
\end{equation}

\noindent Let us now derive Lemma \ref{ll1}.
Let $n\in\Omega_k^0$ and $m \in \Omega_k$.
We first make the decomposition
\begin{equation}
\label{firstdecomp}
E_m(T_n)=(T)_1+(T)_2,
\end{equation}
where

\begin{eqnarray*}
&\ &(T)_1:=
E_m(T_n
1_{T_\oplus<T_n})\qquad \text{and}\\
&\ &(T)_2:=E_m(T_n 1_{T_\oplus>
T_n}).
\end{eqnarray*}
We also denote by $\overline{(T)_2}$ the supremum of $(T)_2$ over all possible $n\in\Omega_k^0$ and $m \in \Omega_k$.
 Now,

\begin{equation}
\label{t2}
(T)_2=(T)_{21}+(T)_{22},
\end{equation}
where

\begin{eqnarray*}
&\ &(T)_{21}:=
E_m(T_n
1_{T_\oplus>T_n}1_{Z_{k}(1)> CN})\qquad\text{and}\\
&\ &(T)_{22}:=E_m(T_n 1_{T_\oplus>
T_n}1_{Z_{k}(1)\le CN}).
\end{eqnarray*}
Now note that for any constant $C<p_k$, by the
Markov property and a standard large deviation estimate we have that
\begin{eqnarray} \nonumber
(T)_{22}
&=&P_m(T_n=1) + \sum_{z_1\le CN} E_m(T_n 
1_{T_\oplus>
T_n \geq 2} 1_{Z_{k}(1)=z_1})\\  \nonumber
&	\le& P_m( Z(1)= n) +
\left(1+\overline{(T)_2}\right)P_m(Z_{k}(1)\!\le \!CN, Z(1)\neq n)\\ \nonumber
&	\le& \left(1+\overline{(T)_2}\right)P_m(Z_{k}(1)\!\le \!CN)\\ \label{eq:bozo}
&\le & 
\left(1+\overline{(T)_2}\right) e^{-cN},
\end{eqnarray}
for some constant $c>0$ depending on $C, p_k$. On the other hand,
by definition of the event $\{ T_\oplus>T_n\}$, we have the first equality below:
\begin{eqnarray} \nonumber
(T)_{21} &=& 
E_m(T_n 1_{T_\oplus> T_n}1_{Z_{k}(1)> CN}1_{Z_{k}(2)\le N-1})\\  \nonumber
&\le& 
\left(1-(1-(1-p_k)^{CN})^N\right)(2+\overline{(T)_2})\\  \label{eq:boum}
&\le& C'N(1-p_k)^{CN}(2+\overline{(T)_2}),
\end{eqnarray}
for some $C'>0$. We can now conclude from (\ref{t2}), (\ref{eq:bozo}) and  (\ref{eq:boum}),
 that there is a constant
$C>0$ such that

$$
\overline{(T)_2}\le Ce^{-CN}.
$$
Let us now take $m=n \in \Omega_k^0$ and  examine the first term of the decomposition
(\ref{firstdecomp}). Note that by the strong Markov property,

\begin{equation}
\label{t1}
(T)_1=E_\oplus(T_n)
+E_n(T_\oplus 1_{T_\oplus
<T_n}).
\end{equation}
Now, by part $(a)$ Lemma \ref{bound-moment} with $p=1$,
we see that the second term in the
above decomposition is bounded above as follows,
\begin{equation}
\label{eq:rrr}
E_n(T_\oplus)\leq 2(1+e^{-CN}).
\end{equation}
Collecting our estimates, we get
\begin{eqnarray*}
E_n T_n &=& 
E_n( T_n; T_\oplus < T_n) + E_n( T_n; T_n < T_\oplus ) \\
&=&  
E_n( T_\oplus;  T_\oplus < T_n) + E_n( T_n-T_\oplus; T_\oplus < T_n) +
E_n( T_n; T_n < T_\oplus ) \\
&=&  
E_n( T_\oplus;  T_\oplus < T_n) + P_n( T_\oplus < T_n) \times E_\oplus( T_n) +
E_n( T_n; T_n < T_\oplus ).
\end{eqnarray*}
Here we bound the first term with (\ref{eq:rrr}), the last one by $\overline{(T)_2}$, and we can use 
(\ref{eq:xyzt}) to obtain the desired conclusion. \qed
\subsubsection{ Proof of Lemma \ref{red-bern}} 

First note that
$$
\sum_{n\in\Omega_k^0}P_\oplus(T_n<T_\oplus)
\geq P_\oplus(T_A<T_\oplus),
$$
and it suffices to prove an inequality in the converse direction.
It is natural to introduce the number ${\mathcal N}_A$ of visits 
of the chain to the set $ \Omega_k^0$ before reaching the $\oplus$ state,  
$${\mathcal N}_A:=\sum_{t= 1}^{T_\oplus}{\bf 1}_{Z(t) \in \Omega_k^0}\;,$$  since we have, for all $m \in \Omega_k$, the relations
\begin{equation} \label{eq:jour0}
E_m {\mathcal N}_A \geq \sum_{n\in\Omega_k^0}P_m(T_n<T_\oplus)\;,\quad
P_m({\mathcal N}_A \geq 1)=P_m(T_A<T_\oplus)\;.
\end{equation}

Then, by the strong Markov property,
\begin{eqnarray} 
E_\oplus( {\mathcal N}_A) 
&=&
E_\oplus( {\mathcal N}_A 1_{{\mathcal N}_A\geq 1})  \nonumber \\
&=& 
\sum_{n \in \Omega_k^0} 
E_\oplus\left(   1_{T_A<T_\oplus, Z(T_A)=n} E_n(1+ {\mathcal N}_A )
\right)  \nonumber  \\
&\leq& 
\left(1+\sup_{n \in \Omega_k^0} E_n({\mathcal N}_A ) \right)
P_\oplus({\mathcal N}_A \geq 1 )\;.
 \label{eq:jour1}
\end{eqnarray}
In view of (\ref{eq:jour0}), where the first term is smaller than the last one, 
it suffices to show that 
$$
\sup_{n \in \Omega_k^0} E_n({\mathcal N}_A ) = {\mathcal O}(e^{-CN})
$$
in order to conclude the proof of the Lemma. In this purpose, use the strong Markov property to write
\begin{eqnarray} 
E_n( {\mathcal N}_A) 
&=&
E_n\left( {\mathcal N}_A 1_{T_\oplus =1} \right) +
E_n\left( {\mathcal N}_A 1_{T_\oplus  \geq 2} \right)  \nonumber \\
&=&
0+
\sum_{m \in \Omega_k^0} 
E_n \left(  
1_{T_A<T_\oplus, Z(T_A)=m}
(1+ E_m {\mathcal N}_A)
\right)  \nonumber \\
&\leq& 
\left(1+\sup_{m \in \Omega_k} E_m({\mathcal N}_A ) \right)
P_n( T_A<T_\oplus)\;.
 \label{eq:jour2}
\end{eqnarray}
Observe also that, for all $n \in \Omega_k$,
\begin{eqnarray} 
 P_n( T_A<T_\oplus)
&\leq&  P_n( T_A=1) +  P_n( T_\oplus>2)
  \nonumber \\
&\leq& (1-p_k)^N+ \sup_{n \in \Omega_k} P_n( T_\oplus>2)  \nonumber \\
&\leq& 2 e^{-CN}
 \label{eq:jour3}
\end{eqnarray}
by (\ref{upb}). Now, the desired result follows from (\ref{eq:jour2}) and (\ref{eq:jour3}),
provided that the supremum in the former estimate is finite. To show this, note that
$ \sup_m P_m( T_\oplus \geq 2 ) \leq (1-p_k)^N $, which implies that
$T_\oplus $ is stochastically dominated by a geometric variable with this parameter. 
Therefore,
 $$ \sup_m E_m ( {\mathcal N}_A ) \leq \sup_m E_m ( T_\oplus )
\leq (1-p_k)^{-N},$$
ending the proof. \qed

\subsection{Proof of Theorem \ref{th:gap}}
\label{proof {th:gap}}
Changing $\xi$ into $(\xi - a)/(a-b)$, we can restrict to the case $a=0, b=-1$. 
Then, for fixed $\eps>0$, we define i.i.d. sequences $\hat \xi_{i,j}(t)$ and $\breve \xi_{i,j}(t)$ by
$$
\hat \xi_{i,j}(t)= -{\bf 1}_{\{\xi_{i,j}(t) \leq -1\} }\;,\qquad 
\breve \xi_{i,j}(t)=  (1+\eps) \sum_{\ell \leq -1} 
\ell {\bf 1}_{\{\xi_{i,j}(t) \in [\ell (1+\eps), (\ell+1)(1+\eps))\}} \;.
$$
Clearly, these variables are integrable since $\xi$ is.
Since $\breve \xi_{i,j}(t)\leq \xi_{i,j}(t) \leq \hat \xi_{i,j}(t)$, the corresponding speeds
are such that 
$$
\breve v_N \leq v_N \leq \hat v_N \;.
$$
From Theorem \ref{th:multinomial}, both $\hat v_N$ and $(1+\eps)^{-1}\breve v_N$ are 
$-(1-p)^{N^2}2^N + o((1-p)^{N^2}2^N)$ as $N \to \8$, which, in addition to the previous 
inequalities, yields
$$
-(1+\eps) \leq  \liminf_{N \to \8}  v_N (1-p)^{-N^2}2^{-N} \leq
\limsup_{N \to \8}  v_N (1-p)^{-N^2}2^{-N} \leq -1.
$$
Letting $\eps \searrow 0$, we obtain the desired claim. \qed

\bigskip




\end{document}